\documentclass[11pt]{amsart}
\usepackage{amsmath, amsfonts, amssymb,amsthm}
\usepackage{euscript}
\usepackage[T1]{fontenc}
\usepackage{graphicx}
\usepackage{color}

\newtheorem{thm}{Theorem}[section]
\newtheorem{lem}[thm]{Lemma}
\newtheorem{prop}[thm]{Proposition}
\newtheorem{cor}[thm]{Corollary}

\theoremstyle{definition}
\newtheorem{defn}[thm]{Definition}

\newtheorem{ex}[thm]{Example}
\newtheorem{rem}[thm]{Remark}

\DeclareMathOperator{\R}{\mathbb R}

\DeclareMathOperator{\N}{\mathbb N}
\DeclareMathOperator{\Ze}{\mathbb Z}
\DeclareMathOperator{\un}{\mathbf 1}
\DeclareMathOperator{\Z}{\mathcal Z}
\DeclareMathOperator{\I}{\mathcal I}

\DeclareMathOperator{\SR}{\mathcal R}
\DeclareMathOperator{\BR}{\mathcal B}
\DeclareMathOperator{\AR}{\mathcal AR}
\DeclareMathOperator{\Co}{\mathcal C}
\DeclareMathOperator{\SO}{\mathcal O}
\DeclareMathOperator{\Pol}{\mathcal P}
\DeclareMathOperator{\K}{\mathcal K}

\DeclareMathOperator{\codim}{codim}
\DeclareMathOperator{\poldepth}{pol-depth}

\DeclareMathOperator{\Bd}{Bd}

\DeclareMathOperator{\Sing}{Sing}

\DeclareMathOperator{\dom}{dom}

\DeclareMathOperator{\Sp}{Spec}
\DeclareMathOperator{\Witt}{W}
\DeclareMathOperator{\sign}{sign}
\DeclareMathOperator{\F}{F}
\DeclareMathOperator{\A}{A}
\DeclareMathOperator{\NA}{N}
\DeclareMathOperator{\AS}{AS}
\DeclareMathOperator{\w}{w}

\DeclareMathOperator{\length}{\ell}
\DeclareMathOperator{\pol}{indet}

\def \S {{\mathcal S}}

\begin{document}

\title[\tiny{Semi-algebraic geometry with
  rational continuous functions}]{Semi-algebraic geometry with rational continuous functions}

\date{\today}

\author[J.-P.~Monnier]{Jean-Philippe Monnier}

\address{Jean-Philippe Monnier\\
   LUNAM Universit\'e, LAREMA, Universit\'e d'Angers}
\email{jean-philippe.monnier@univ-angers.fr}


\maketitle

\begin{quote}\small
\end{quote}

\begin{abstract}
Let $X$ be a real algebraic subset of $\R^n$. We investigate on the
theory of algebraically constructible functions on $X$ and the
description of the semi-algebraic subsets of $X$ when we replace the
polynomial functions on $X$ by some rational continuous functions on
$X$.
\end{abstract}

\section{Introduction}

The concept of rational continuous maps between smooth real algebraic sets
was used the first time by W. Kucharz \cite{Ku} in order to
approximate continuous maps into spheres. In \cite{FHMM}, rational continuous
functions on smooth real algebraic sets are renamed by ``regulous
functions'' and their systematic study is performed.
A theory of vector bundles using these functions is done in
\cite{KuKu1}. They also appear in the recent theory of
piecewise-regular maps \cite{Ku2}.

J. Koll\'ar, K. Nowak \cite[Prop. 8]{KN}  \cite[Thm. 4.1]{FHMM} proved that
the restriction of a regulous function to a real algebraic subset is
still rational (this can also be deduced from \cite[Thm. 4.1]{FHMM}). It allows us to define the concept of regulous
function on a possibly singular real algebraic set $X$ by
restriction from the ambiant space. On $X$, we have two classes of
functions: rational continuous functions and regulous functions. In
cite \cite{KN} and \cite{KuKu2}, they give conditions for
a rational continuous function to be regulous. In the second section
of the present paper we present some preliminaries and we continue the study of differences between these
two classes of functions.

In classical real algebraic geometry, we copy what happens in the
complex case, and so we use as sheaf of functions on a real algebraic
variety the sheaf of regular functions. Unfortunately and contrary
to the complex case, some defects appear: classic Nullstellensatz and
theorems A and B of Cartan are no longer valid \cite{BCR}.
In \cite{FHMM}, G. Fichou, J. Huisman,
F. Mangolte, the author show that the use of the sheaf of regulous functions
instead of the
sheaf of regular functions corrects these defects. In this paper, and
from the third section, we
do the same thing but now in the semi-algebraic framework, we
introduce
a regulous semi-algebraic geometry i.e a semi-algebraic geometry with
regulous functions replacing polynomial or regular functions (remark that a regulous
function is semi-algebraic). The aim of \cite{FHMM} was to study the
zero sets of regulous functions, our purpose here is to investigate on their signs.

The third section deals with the theory of algebraically constructible functions, due to C. McCrory
and A. Parusi\'nski \cite{MP1}. This theory has been developed to study singular real algebraic
sets. We prove that the theory of
algebraically constructible functions can be done using only regulous
objects (functions, maps, sets). In particular,\\
\\
{\bf  Theorem A.}\\
Let $X\subset\R^n$ be a real algebraic set.
The sign of a regulous function on $X$ is a sum of
signs of polynomial functions on $X$. In particular, the algebraically
constructible functions on $X$ are exactly the sum of signs of
regulous functions on $X$.\\
\\

In the fourth and sixth sections, we investigate on the number of polynomial functions needed in
the representation of Theorem A. This is connected to the work of I. Bonnard in
\cite{Bo1} and \cite{Bo2}. We also study the case where the sign of a
regulous function is the sign of a polynomial function.\\
\\
{\bf  Theorem B.}\\
Let $X\subset\R^n$ be a real algebraic set and let $f$ be a regulous
function on $X$.
The sign of $f$ on $X$ coincides with the 
sign of a polynomial functions on $X$ if and only if the zero set of
$f$ is Zariski closed.\\

In the fifth section, we focus on the description of principal
semi-algebraic sets when we replace polynomial functions by regulous
functions. We compare regulous principal semi-algebraic
sets and polynomial principal semi-algebraic sets. This comparison is
useful to get Theorem B. In particular,\\
\\
{\bf  Theorem C.}\\
Let $X\subset\R^n$ be a real algebraic set. Let $f$ be a regulous
function on $X$ and we denote by $S$ the regulous principal open
semi-algebraic set $\{x\in X|\,f(x)>0\}$. Then $S$ is a principal open
semi-algebraic set, i.e the exists a polynomial function $p$ on $X$
such that $S=\{x\in X|\,p(x)>0\}$ if and only if
$S\cap\overline{\Bd(S)}^{Zar}=\emptyset$ where
$\overline{\Bd(S)}^{Zar}$ denote the Zariski closure of the euclidean
boundary of $S$.\\

In the last section, we characterize the signs of continuous semi-algebraic functions that coincide with the signs of regulous functions.
\\
\\
{\it Acknowledgements}

I want to thank G. Fichou and R. Quarez for stimulating conversations
concerning regulous functions. I thank also F. Mangolte and D. Naie
for interesting and helpful discussions. I thank a referee of an
earlier version of the paper for pointing out to us a mistake in this
earlier version.
In memory of J.-J. Risler.




\section{Regulous functions versus rational continuous functions}

\subsection{Regulous functions}

Let $n\in\N$ and $k\in\N\cup\{\infty\}$, we recall the definition of
$k$-regulous functions on $\R^n$ (see \cite{FHMM}).

\begin{defn}
\label{def1Rn}
We say that a function $f:\R^n\to \R$ is $k$-regulous on $\R^n$ if $f$ is $C^k$ on $\R^n$  and $f$ is a rational function on $\R^n$, i.e. there exists a
non-empty Zariski open subset $U\subseteq \R^n$ such that $f|_U$ is
regular.\par
A $0$-regulous function on $\R^n$ is simply called a regulous function
on $\R^n$.
\end{defn}

An equivalent definition of a
$k$-regulous function on $\R^n$ is given in \cite[Thm. 2.15]{FMQ}.

We denote by
$\SR^k(\R^n)$ the ring of $k$-regulous functions on $\R^n$. By Theorem 3.3
of \cite{FHMM} we know that $\SR^{\infty}(\R^n)$ coincides with the ring
$\SO (\R^n)$ of regular functions on $\R^n$.

For an
integer $k$, 
the $k$-regulous topology of $\R^n$ is defined to be the topology whose closed subsets are generated by the zero sets of regulous functions in $\SR^k (\R^n)$. 
Although the $k'$-regulous topology is a priori finer than the $k$-regulous topology when $k'<k$, it has been proved in \cite{FHMM} that in fact they are the same.
Hence, it is not necessary to specify the integer $k$ to define the
regulous topology on $\R^n$. By \cite[Thm. 6.4]{FHMM}, the regulous
topology on $\R^n$ is the algebraically constructible topology on
$\R^n$ (denoted by $\Co$-topology). On $\R^n$, the euclidean topology is finer than the
$\AR$-topology (the arc-symmetrical topology introduced by K. Kurdyka \cite{KK}) which is
finer than the regulous topology (see \cite{FHMM}) which is the $\Co$-topology which is finer than
the Zariski topology.

We give now the definition of a regulous function on a real
algebraic set \cite[Cor. 5.38]{FHMM}. We recall that in real
algebraic geometry, when we focus only on real points then we are concerned almost exclusively with affine
varieties (see \cite[Rem. 3.2.12]{BCR}) and thus with real algebraic sets.
\begin{defn}
\label{def1X}
Let $X$ be a real algebraic subset of $\R^n$.
A $k$-regulous function on $X$ is the
restriction to $X$ of a $k$-regulous function on $\R^n$. The ring of
$k$-regulous functions on $X$, denoted by $\SR^k(X)$, corresponds to $$\SR^k(X)=\SR^k(\R^n)/\I_k(X)$$
where $\I_k(X)$ is the ideal of $\SR^k(\R^n)$ of $k$-regulous functions
on $\R^n$ that vanish identically on $X$.
\end{defn}

\begin{rem} In \cite{FHMM} the previous definition is extended to the case $X$
  is a closed regulous subset of $\R^n$.
\end{rem}

Recall that a real function on a semi-algebraic set is called
semi-algebraic if its graph is a semi-algebraic set.
\begin{prop}
\label{s-a1}
Let $X$ be a real algebraic subset of $\R^n$. A regulous function on
$X$ is a semi-algebraic function.
\end{prop}

\begin{proof}
Let $f\in\SR^0(X)$. By definition, $f$ is the restriction to $X$ of a
regulous function $\hat{f}\in\SR^0(\R^n)$. The function $f$ is
semi-algebraic since $\hat{f}$ is semi-algebraic
\cite[Prop. 3.1]{FHMM}.
\end{proof}

Let $X\subset \R^n$ be a real algebraic set, we will denote by $\SO (X)$
the ring of regular functions on $X$, by $\Pol (X)$ the ring of
polynomial functions on $X$ and by $\K (X)$ the ring of rational
functions on $X$.
By \cite[Prop. 8]{KN} or \cite[Thm. 4.1]{FHMM}, a regulous function on $X$
is always rational on $X$ (coincides with a regular function on a dense
Zariski open subset of $X$).
Since the regulous topology on $X$ is
sometimes strictly finer than the Zariski topology on $X$, the ring
$\SR^0(X)$ is not always a subring of $\K(X)$ even if $X$ is Zariski
irreducible. We will denote by $\Z(f)$ the zero set of a real function $f$
on $X$.

\begin{ex}
\label{exemple1}
{\rm Let $X$ be the plane cubic with an isolated point
$X=\Z(x^2+y^2-x^3)$. The curve $X$ is Zariski irreducible but $\Co$-reducible. The
$\Co$-irreducible components of $X$ are $F$ and $\{(0,0)\}$ where
$F=\Z(f)\subset \R^2$, with $f=1-\dfrac{x^3}{x^2+y^2}$ extended
continuously at the origin, is
the smooth branch of $X$. 
  \begin{figure}[ht]
\centering
\includegraphics[height =4cm]{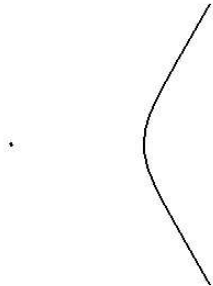}
\caption{Cubic curve with an isolated point.}
        \label{fig.cubic}
\end{figure}
The ring
 $\SR^0(X)$ is the cartesian product $\SR^0(F)\times \R$ and the
 class of $f$ in $\SR^0(X)$ is $(0,1)$. Remark that the ring
 $\SR^0(X)$ is not an integral domain and consequently it is not a
 subring of $\K(X)$.}
\end{ex}

Let $X$ be a real algebraic subset of $\R^n$. Let $f\in\K (X)$ and let
$U$ be a dense Zariski open subset of
$X$, we say that the couple $(U,f|_U)$ or the function $f|_U$ is
a regular presentation of $f$ if $f|_U$ is regular.
We have a natural ring morphism
$\phi^0:\SR^0(X)\rightarrow \K(X)$ which send $f\in \SR^0(X)$ to the
class $(U,f|_U)$ in $\K(X)$, where $(U,f|_U)$ is a regular
presentation of $f$.
We have seen that $\phi^0$ is not always injective. 

\begin{defn}
Let $X$ be a real algebraic subset of $\R^n$. Let $f\in \K(X)$. We say
that the rational function $f$ can be extended continuously to $X$ if
there exists a regular presentation $f|_U$ of $f$ that can be extended
continuously to $X$.
\end{defn}

In the following,
we will denote by $\overline{E}^{\tau}$ the closure of the subset $E$
of $\R^n$ for the topology $\tau$ on $\R^n$. We prove now that $\phi^0$ is injective
if and only if $\overline{X_{reg}}^{\Co}=X$,
$X_{reg}$ denoting the smooth locus of
$X$. If $X$ is irreducible then the condition $\overline{X_{reg}}^{\Co}=X$ means that $X$ is also irreducible for the
$\Co$-topology (see \cite{FHMM}).

\begin{lem}
\label{contreg}
Let $X$ be a real algebraic subset of $\R^n$. Let $U$ be a dense Zariski open subset of
$X$. Then $X_{reg}\subset \overline{U}^{eucl}$.
\end{lem}

\begin{proof}
Without loss of generality we can assume $X$ is irreducible. Let $Z$
denote the Zariski closed set $X\setminus U$. Assume $x\in
X_{reg}\setminus \overline{U}^{eucl}$. So there exists an open
semi-algebraic subset $U'$ of $X$ such that $x\in U'$ and $U'\subset
X\setminus \overline{U}^{eucl}\subset Z$. Hence $\dim U'\leq \dim
Z<\dim X$, this is impossible by \cite[Prop. 7.6.2]{BCR}.
\end{proof}

\begin{prop}
\label{inclusion}
Let $X$ be a real algebraic subset of $\R^n$.
The map $\phi^0:\SR^0(X)\rightarrow \K(X)$ is injective if and only if
$\overline{X_{reg}}^{\Co} =X$.
\end{prop}

\begin{proof}
Assume $\overline{X_{reg}}^{\Co} =X$. Let $f_1,f_2\in\SR^0(X)$ be such that $\phi^0(f_1)=\phi^0(f_2)$. Let
$\hat{f}_i\in\SR^0(\R^n)$, $i=1,2$, be such that $
\hat{f}_i|_X=f_i$. Since $f_1$ and $f_2$ are two continuous extensions
to $X$ of the same rational function on $X$, they coincide on
$X_{reg}$ by Lemma \ref{contreg}. Hence $\hat{f}_1-\hat{f}_2$ vanishes on $X$ since $X$ is the regulous
closure of $X_{reg}$. It implies that $f_1=f_2$.

Assume $\overline{X_{reg}}^{\Co} \not=X$. By \cite[Thm. 6.13]{FHMM},
we may write $X=\overline{X_{reg}}^{\Co} \cup F$ with $F$ a non-empty
regulous closed subset of $\R^n$ such that $\dim F<\dim X$. Let
$\hat{f}\in\SR^0(\R^n)$ be such that
$\Z(\hat{f})=\overline{X_{reg}}^{\Co}$ and let $f$ denote the
restriction of $\hat{f}$ to $X$. We have $f\not= 0$ in $\SR^0(X)$,
$\phi^0(f)=0$ in $\K(X)$ and thus $\phi^0$ is non injective.
\end{proof} 

\subsection{Rational continuous functions on central real algebraic sets}

Let $n$ be a positive integer and let $X\subset \R^n$ be a
real algebraic set. Let $f\in \K(X)$ be a rational function on
$X$. The domain of $f$, denoted by $\dom (f)$, is the biggest dense Zariski open subset of $X$ on
which $f$ is regular, namely $f=\dfrac{p}{q}$ on $\dom (f)$ where $p$ and $q$ are polynomial
functions on $\R^n$ such that $\Z(q)=X\setminus\dom(f)$ (see
\cite[Prop. 2.9]{FHMM}). The
indeterminacy locus or polar locus of $f$ is defined to be the Zariski closed set
$\pol(f)=X\setminus \dom(f)$. By definition, $\dim \pol(f)<\dim X$.

\begin{defn}
\label{defratcont}
Let $X$ be a real algebraic subset of $\R^n$. Let $f$ be a real continuous function on $X$. We say that
$f$ is a rational continuous function on $X$ if $f$ is rational on $X$
i.e there exists a
dense Zariski open subset $U\subseteq X$ such that $f|_U$ is
regular. 
\end{defn}

\begin{rem} We may also define a rational continuous function as a
  continuous extension of a rational function.
\end{rem}

Let $\SR_0(X)$ denote the ring of rational continuous
functions on $X$.
We have a natural ring morphism
$\phi_0:\SR_0(X)\rightarrow \K(X)$ which send $f\in \SR_0(X)$ to the
class $(U,f|_U)$ in $\K(X)$, where $(U,f|_U)$ is a regular
presentation of $f$.

\begin{rem} We have $\SR_0(\R^n)=\SR^0(\R^n)$.
\end{rem}

\begin{defn} 
\label{central}
We say that $X$ is ``central'' if $\overline{X_{reg}}^{eucl}=X$.
\end{defn}

\begin{rem} The previous definition comes from the introduction of the
  the central locus of a real algebraic set made in
  \cite[Def. 7.6.3]{BCR}. By \cite[Prop. 7.6.2]{BCR}, an irreducible
  real algebraic set $X$ is central if and only if the dimension at
  any point of $X$ is maximal.
\end{rem}

The property to be central is a property of an algebraic set that
ensures a rational continuous function on it to be the unique possible
continuous extension of its associated rational function. It also
ensures that rational continuous functions are
semi-algebraic functions.
The following example illustrates these facts.

\begin{ex}
\label{exemple4}
{\rm Let $X=\Z(zx^2-y^2)\subset\R^3$ be the Whitney umbrella.
\begin{figure}[ht]
\centering
\includegraphics[height =6cm]{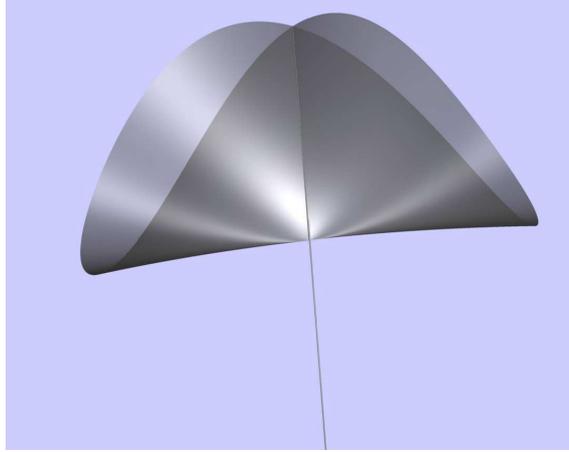}
\caption{Whitney umbrella.}
        \label{fig.whitney}
\end{figure}
By \cite{FHMM}, $X$ is irreducible in the $\Co$-topology and
we have $$
 \overline{X_{reg}}^{\AR}
=   \overline{X_{reg}}^{\Co}=\overline{X_{reg}}^{Zar}=X.$$
The set $X\setminus\overline{X_{reg}}^{eucl}$ is the half of the
stick. The function $\dfrac{y^2}{x^2}|_X$ is regular on $X$ outside of
the stick and so it
gives rise of a rational function on $X$. 
Its class in $\K(X)$ is also the class of the regular function $z|_X$
($(X\setminus\Z(x^2+y^2),\dfrac{y^2}{x^2}|_{X\setminus\Z(x^2+y^2)})$
and $(X,z|_X)$ are two regular presentations of the same rational function). This rational function
can be extended continuously in many different ways to $X$: we can extend
the regular presentation $(X\setminus\Z(x^2+y^2),\dfrac{y^2}{x^2}|_{X\setminus\Z(x^2+y^2)})$
by $z$ on $X\cap \Z(x^2+y^2)$ (we get the regular function $z|_X$
on $X$) but we can also
extend it by $z$ on $X\cap \Z(x^2+y^2)\cap \{z\geq 0\}$ and by $\sin z$ on
$X\setminus\overline{X_{reg}}^{eucl}=X\cap \Z(x^2+y^2)\cap \{z<
0\}$. The first extension is the unique
regulous extension to $X$ of $\dfrac{y^2}{x^2}$ (Proposition
\ref{inclusion}) and the second one is
a non-regulous rational
continuous function on $X$ that is not semi-algebraic. Consequently, the map
$\phi_0:\SR_0(X)\rightarrow \K(X)$ is not injective.}
\end{ex}

\begin{prop}
\label{inclusion2}
Let $X$ be a real algebraic subset of $\R^n$. The map $\phi_0:\SR_0(X)\rightarrow \K(X)$ is injective if and only if $X$
is central.
\end{prop}

\begin{proof}
Under the hypothesis $X=\overline{X_{reg}}^{eucl}$, it follows from
Lemma \ref{contreg} that if a rational
function of $\K(X)$ has a continuous extension to $X$ then this
extension is the unique possible continuous extension. 

Assume $X$ is not central. It is always possible to extend the null
function on $\overline{X_{reg}}^{eucl}$ to a continuous function $f$
on $X$ such that $f$ is not the null function on $X$. The function $f$
is rational on $X$ since it has a regular presentation on
$X_{reg}$. Obviously, $f$ is a non-trivial element of the kernel of
$\phi_0$ and the proof is done.
\end{proof}

\begin{prop}
\label{s-a2}
Let $X$ be a central real algebraic subset of $\R^n$. The rational continuous
functions on $X$ are semi-algebraic functions.
\end{prop}

\begin{proof}
Let $f\in\SR_0(X)$ and let $(U,f|_U)$ be a regular presentation of
$f$. It is clear that $f|_U$ is a semi-algebraic function (on $U$).
By Lemma \ref{contreg}, the graph of $f$ is the euclidean closure of
the graph of $f|_U$. The function $f$ is semi-algebraic by
\cite[Prop. 2.2.2]{BCR}.
\end{proof}

\begin{rem}
There exist non-central real algebraic sets for which the rational
continuous functions are always semi-algebraic: Consider the
non-central real algebraic
set $X$ of Example \ref{exemple1}. By Corollary \ref{courbe} and
Proposition \ref{s-a1}, a rational continuous function on $X$ is
semi-algebraic. More generally, it is not difficult to prove that: all the
rational continuous functions on a real algebraic set $X$ are
semi-algebraic if and only if $\dim (X\setminus
\overline{X_{reg}}^{eucl})<1$.
\end{rem}

In the following, to simplify notation, we sometimes identify a
rational continuous function on a central real algebraic set with one
of its regular presentations
(e.g. $\dfrac{x^3}{x^2+y^2}\in\SR^0(\R^2)$).
By \cite[Prop. 8]{KN} or \cite[Thm. 4.1]{FHMM}, any $f\in \SR^0(X)$ can be identified with a unique function in
$\SR_0(X)$. Hence we get:
\begin{prop}
\label{inclusion1}
Let $X$ be a real algebraic subset of $\R^n$. We have the following ring inclusion $\phi^0_0:\SR^0(X)\hookrightarrow
\SR_0(X)$ and moreover
$$\phi^0=\phi_0\circ \phi^0_0.$$
\end{prop}

\begin{rem} Let $X$ be a real algebraic subset of $\R^n$ such that
  $\overline{X_{reg}}^{\Co} =X$ and $X$ is not central (e.g the
  Whitney umbrella). By Propositions \ref{inclusion1} and
  \ref{inclusion2}, we see that in this case the map $\phi^0_0$ is not
  surjective i.e there is a rational continuous function on $X$ which
  is not regulous.
\end{rem}

In the following example, due to Koll\'ar and Nowak \cite[Ex. 2]{KN}, we will see that,
even if $X$ is central,
$\phi^0_0$ may be not surjective.

\begin{ex}
\label{exemple5}
{\rm Let $X=\Z(x^3-(1+z^2)y^3)\subset \R^3$. Then $X$ is a central singular
  surface with singular locus the $z$-axis. 
By \cite[Ex. 2]{KN}, the class of the rational fraction $\dfrac{x}{y}|_X$ in $\K(X)$
can be extended continuously to $X$ (in a unique way) by the function
$(1+z^2)^{\frac{1}{3}}$ on the $z$-axis and gives an element
$f\in\SR_0(X)$. Moreover, $f$ cannot be extended to an element of
$\SR_0(\R^3)=\SR^0(\R^3)$ (the reason is that the restriction of $f$ to the
$z$-axis $(1+z^2)^{\frac{1}{3}}$ is not rational) and thus $f$ is not in $\SR^0(X)$. Here the map $\phi^0_0:\SR^0(X)\hookrightarrow
\SR_0(X)$ is not surjective and the map $\phi_0:\SR_0(X)\rightarrow \K(X)$
is injective.}
\end{ex}

One of the goal of the paper \cite{KN} was to study the surjectivity of the
map $\phi^0_0$ when $X$ is a central real algebraic set. Notice that
``regulous functions'' are named ``hereditarily rational continuous
functions'' in \cite{KN}.

We reformulate with our notation the three principal results of \cite{KN}
with an improvement of the first one.

The following lemma can be obtained from the arguments used in the
proof of \cite[Prop. 11]{KN}.
\begin{lem} (proof of \cite[Prop. 11]{KN})\\
\label{extensionKN} Let $X\subset \R^n$ be a real algebraic set and let $f\in \SR_0(X)$. Let $W=\pol(f)$
be the polar locus of $f$ in $X$. If $f|_W\in \SR^0(W)$  has the additional
property that $f|_W$ is the restriction to $W$ of $g\in\SR^0(\R^n)$ such
that $g$ is regular on $\R^n\setminus W$ then $$f\in\SR^0(X).$$
Moreover, $f$ has also the additional property that there exists $\hat{f}\in\SR^0(\R^n)$ such
that $\hat{f}$ is regular on $\R^n\setminus W$ and $\hat{f}|_X=f$.
\end{lem}

We improve Lemma \ref{extensionKN} by removing the additional property
from the hypotheses.
\begin{lem}
\label{extension} Let $X\subset \R^n$ be a real algebraic set and let $f\in \SR_0(X)$. Let $W=\pol(f)$
be the polar locus of $f$ in $X$. If $f|_W\in \SR^0(W)$ then $$f\in\SR^0(X).$$
\end{lem}

\begin{proof}
Assume $f|_W\in
\SR^0(W)$. By definition, there exists $g\in\SR^0(\R^n)$ such that
$g|_W=f|_W$. We denote by $g_0$ the regulous function $g|_W$.
We consider the
following sequence of regulous functions 
$$(g_0,g_1=(g_0)|_{\pol(g_0)},g_2=(g_1)|_{\pol(g_1)},\ldots)$$
on a sequence of Zariski closed subsets $(W_i=\pol(g_{i-1}))$ of $W$ of dimension strictly
decreasing and included one in another. The functions $g_i$ are regulous
since they are also a restriction of a regulous function on $\R^n$.
We claim that there exists an integer $m$
such that $g_m$ is a regular function on $W_m$. Indeed, $g_m$ is
automatically regular if $\dim W_m=0$
and we get the claim since $\dim W_{i+1}<\dim W_i$. By
\cite[Prop. 3.2.3]{BCR}, $g_m$ is the restriction to $W_m=\pol (g_{m-1})$ of regular
function $\hat{g}_m$ on $\R^n$. By Lemma \ref{extensionKN} for $f=g_{m-1}$,
$X=W_{m-1}$ and $W=W_m$, we get that $g_{m-1}$ is the restriction to
$W_{m-1}$ of a
regulous function $\hat{g}_{m-1}$ on $\R^n$ regular on
$\R^n\setminus\pol(g_{m-1})$. Repeated application of Lemma
\ref{extensionKN} enables us to see that
$g_0=f|_W$ is the restriction to $W$ of a
regulous function $\hat{g}_{0}$ on $\R^n$ regular on
$\R^n\setminus\pol(f|_W)$. Since $\pol (f|_W)\subset \pol(f)=W$, using one
last time Lemma \ref{extensionKN}, we get the proof.
\end{proof}

\begin{prop} (\cite[Prop. 8]{KN})\\
\label{KNprop8}
Let $X\subset \R^n$ be a real algebraic set and let $f\in\SR_0(X)$.
For any irreducible real algebraic subset $W\subset X$ not contained
in the singular locus of $X$, we have 
$$f|_{W}\in \SR_0(W).$$
\end{prop}

\begin{thm} (\cite[Prop. 8, Thm. 10]{KN})\\
\label{KolNow}
Let $X\subset \R^n$ be a smooth real algebraic
set. Then the map $\phi^0_0:\SR^0(X)\hookrightarrow
\SR_0(X)$ is an isomorphism.
\end{thm}

\begin{proof}
By \cite[Prop. 8]{KN}, a rational continuous function on a smooth real
algebraic set is hereditarily rational. By \cite[Thm. 10]{KN}, a
continuous hereditarily rational function on a non necessary smooth
real algebraic set $X\subset \R^n$ is the restriction of a rational
continuous function on $\R^n$ and thus ``continuous hereditarily
rational'' means ``regulous''.
\end{proof}

We extend the result of Theorem \ref{KolNow} to real algebraic sets
with isolated singularities using Lemma \ref{extension}.

\begin{thm}
\label{isole}
Let $X\subset \R^n$ be a real algebraic set with only isolated singularities. Then 
$$\SR^0(X)=\SR_0(X).$$
\end{thm}

\begin{proof}
Let $f\in \SR_0(X)$. Let $W\subset X$ be a real algebraic subset. If
$\dim W=0$ then $f|_W$ is regular and thus $f|_W\in\SR_0(W)$. If $W$
is irreducible and $\dim W\geq 1$ then $f|_W\in\SR_0(W)$ by
Proposition \ref{KNprop8}. It follows that $f|_W\in\SR_0(W)$ without
hypothesis on $W$. We consider the
following sequence of continuous rational functions 
$$(f_0=f,f_1=f|_{\pol(f)},f_2=(f_1)|_{\pol(f_1)},\ldots)$$
on a sequence of real algebraic subsets $(W_i=\pol(f_{i-1}))$ of $X$ of dimension strictly
decreasing and included one in another. There exists an integer $m$
such that $f_m$ is regular on $W_m$. Using
several times Lemma \ref{extension}, we get that $f\in\SR^0(X)$.
\end{proof}

\begin{cor}
\label{courbe}
Let $X\subset \R^n$ be a real algebraic curve. Then 
$$\SR^0(X)=\SR_0(X).$$
\end{cor}

\subsection{Blow-regular functions and arc-analytic functions on central real algebraic sets} 

In this section, we compare different classes of functions on a
(central) real
algebraic set: regulous functions, rational continuous functions,
blow-regular functions, arc-analytic functions.

By \cite[thm. 3.11]{FHMM}, regulous functions on a smooth real algebraic
set $X\subset \R^n$ coincide with
blow-regular functions on $X$, it gives another equivalent
definition for regulous functions on $X$.

\begin{defn}
\label{def2lisse} 
Let $X\subset\R^n$ be a smooth real algebraic set.
Let $f:X\rightarrow \R$ be a real function.
We say that $f$ is
regular after blowings-up on $X$ or $f$ is blow-regular on $X$ if there exists a 
composition $\pi:M\to X$ of successive blowings-up along smooth
centers such that $f\circ \pi$ is regular on $M$.
We denote by $\BR (X)$ the ring of blow-regular functions of $X$.
\end{defn}

\begin{thm} \cite[thm. 3.11]{FHMM}\\
Let $X\subset\R^n$ be a smooth real algebraic set. We have $\SR^0(X)=\BR(X)$.
\end{thm}

We establish a connection between regulous functions and arc-analytic
functions introduced in \cite{KK}. A function $f:X\rightarrow\R$,
defined on a real analytic variety $X$, is said to be arc-analytic if
$f\circ\gamma$ is analytic for every analytic arc $\gamma:I\rightarrow
X$ where $I$ is an open interval in $\R$.
\begin{prop}
\label{arcana}
Let $X\subset\R^n$ be a real algebraic set. A regulous function on $X$
is arc-analytic.
\end{prop}

\begin{proof}
Let $f\in\SR^0(X)$. By definition, there exists
$\hat{f}\in\SR^0(\R^n)$ such that $\hat{f}|_X=f$. By Proposition
\ref{s-a1} and \cite[thm. 3.11]{FHMM}, $\hat{f}$ is a semi-algebraic
blow-regular function on $\R^n$. It follows from \cite[Thm. 1.1]{BM},
that $\hat{f}$ is an arc-analytic function and therefore $f$ also.
\end{proof}

Now we will give a definition of blow-regular
function on a non-necessarily smooth real algebraic set.

\begin{defn}
\label{defB1}  Let $X\subset\R^n$ be a real algebraic set. Let $\BR
  (X)$ denote the ring of real functions $f$ defined on
  $X$ such that, there exists a resolution of singularities
  $\pi:\tilde{X}\rightarrow X$ (a proper birational regular map such
  that $\tilde{X}$ is smooth) such that the composite $f\circ\pi$ is in
  $\BR(\tilde{X})=\SR^0(\tilde{X})=\SR_0(\tilde{X})$. A $f\in \BR
  (X)$ is called a ``blow-regular function'' on $X$.
\end{defn}

\begin{rem} 
\label{defB2} According to the definition of blow-regular function on a smooth
  variety we get: $f\in \BR
  (X)$ if and only if $f$ is a real function defined on
  $X$ such that there exists a resolution of singularities
  $\pi:\tilde{X}\rightarrow X$ such that $f\circ \pi$ is regular. This
  justifies the notation ``blow-regular''.
\end{rem}

\begin{rem} In the definition \ref{defB1} we can change $\exists$ by
  $\forall$. It is not true in the equivalent definition of the remark
  \ref{defB2}.
\end{rem}

We prove in the following that, even in the central case,
blow-regular functions and rational continuous functions coincide.

\begin{prop}
\label{equivalencesing}
Let $X\subset\R^n$ be a central real algebraic set. We have
$$\BR (X)=\SR_0(X).$$
\end{prop}

\begin{proof}
Assume $f\in\SR_0(X)$ and let $\pi:\tilde{X}\rightarrow X$ be a
resolution of singularities. Then clearly $f\circ\pi$ is rational on
$\tilde{X}$. Since $\pi^{-1}(X)=\tilde{X}$ (see below) then we can
conclude that $f\circ \pi$ is continuous on $\tilde{X}$ and thus
$f\circ \pi\in \SR_0(\tilde{X})$.

Assume $f\in \BR (X)$ and let $\pi:\tilde{X}\rightarrow X$ be a
resolution of singularities. Then $f\circ \pi \in \SR_0(\tilde{X})$
and thus $f$ is rational on $X$. The function $f$ is continuous on $X$
since:\\
$\bullet$ The fibres of $\pi$ are non-empty i.e $\pi$ is surjective.
Indeed if $\pi^{-1}(x)=\emptyset$ for a $x\in X$ then $\dim X_x<\dim
X$ ($\dim X_x$ is the local dimension of $X$ at $x$
\cite[Def. 2.8.12]{BCR}) and thus $x\not\in \overline{X_{reg}}^{eucl}$
by \cite[Prop. 7.6.2]{BCR}, this contradicts our assumption that $X$
is central. The surjectivity of $\pi$ can also be deduced from
\cite[Thm. 2.6, Cor. 2.7]{KK}.\\
$\bullet$ The function $f\circ\pi$ is continuous on $\tilde{X}$.\\
$\bullet$ The function $f\circ\pi$ is
constant on the fibers of $\pi$.\\
In fact, the ``central'' condition forces the strong topology on $X$
to be the quotient topology induced by
the strong topology on $\tilde{X}$. Indeed, $\pi$ is a proper (and
thus closed) surjective map and thus a quotient map.
\end{proof}

The next example illustrates the fact that the assumption that $X$ is central cannot be dropped in the previous
proposition. In general we only have $\SR_0(X)\subset \BR (X)$.
\begin{ex}
{\rm We consider the real algebraic surface introduced in
  \cite[Ex. 6.10]{FHMM}. Let $X$ be the algebraic subset of $\R^4$
  defined by $X=\Z((x+2)(x+1)(x-1)(x-2)+y^2)\cap \Z(u^2-xv^2)$. The
  set $X$ has two connected components $W$ and $Z$, $W$ has dimension
  two and $W=\overline{X_{reg}}^{eucl}$, $Z$ has dimension one and
  $Z=\Z(((x+2)(x+1)(x-1)(x-2)+y^2)^2+u^2+v^2)\cap
  \{(x,y,u,v)\in\R^4|\,x<0\}$. Let $f$ be the real function defined on
  $X$ by $f(x,y,u,v)=\begin{cases}
    \dfrac{1}{(x+1)^2+y^2+u^2+v^2}\,\,{\rm
      if}\,\,(x,y,u,v)\not=(-1,0,0,0)\\
    0\,\,{\rm if}\,\,(x,y,u,v)=(-1,0,0,0) \end{cases}$ Remark that $f$
  is not continuous at the point $(-1,0,0,0)$ and is regular on $W$. Let $\pi:\tilde{X}\rightarrow X$ be a
resolution of singularities. Since $\pi^{-1}(Z)=\emptyset$ then
$f\circ\pi$ will be regular on $\tilde{X}$ and thus $f\in\BR
(X)\setminus\SR_0(X)$.}
\end{ex}

By Propositions \ref{s-a2} and \ref{equivalencesing}, a rational
continuous function on a central real algebraic set is a
semi-algebraic blow-regular function like a regulous function. However,
unlike a regulous function, it is not difficult to see that it can happen that a rational continuous
function on a central real algebraic set is not arc-analytic. The
following example is due to G. Fichou.

\begin{ex}
\label{ratcontnotarcana}
{\rm Consider the function $f=\dfrac{x}{y}$ defined on the real
  algebraic set $X=\Z(x^3-zy^3)\subset\R^3$. Then $X$ is central with
  singular subset given by the $z$-axis. The function $f$ is regular
  outside the $z$-axis and can be extended continuously on the
  $z$-axis by the function $z^{1/3}$. It is clear that the (new)
  function $f$ is not arc-analytic since the image by $f$ of the
  analytic arc $t\mapsto (0,0,t)$ is not analytic. The function $f$
  becomes arc-analytic after resolution of singularities of $X$ by
  Proposition \ref{arcana} but it does not imply that $f$ is also
  arc-analytic. The reason is that some analytic arcs on $X$ can not
  be lifted as analytic arcs when we solve the singularities of $X$.}
\end{ex}

\section{Algebraically constructible functions}

We make reminders on the theory of constructible and algebraically
constructible functions due to C. McCrory and A. Parusi\'nski (see
\cite{MP1}, \cite{MP2}). This theory was remarkably used to study the topology of singular
real
algebraic sets. We
follow the definitions and the results given in \cite{C}.

Let $S$ be a semi-algebraic set. A constructible function on $S$ is a
function $f:S\rightarrow \Ze$ that can be written as a finite sum
$$\varphi=\sum_{i\in I} m_i \un_{S_i}$$
where for each $i\in I$, $m_i$ is an integer and $\un_{S_i}$ is the
characteristic function of a semi-algebraic subset $S_i$ of $S$.
The set of constructible functions on $S$ provided with the sum and
the product form a commutative ring denoted by $\F(S)$.
If $\varphi=\sum_{i\in I} m_i \un_{S_i}$ is a constructible function
then the Euler integral of $\varphi$ on $S$ is 
$$\int_S\varphi\, d\chi=\sum_{i\in I} m_i \chi(S_i)$$
where $\chi$ is the Euler characteristic with compact support.
Let $f:S\rightarrow T$ be a continuous semi-algebraic map between
semi-algebraic sets and
$\varphi\in \F(S)$. The pushforward $f_*\varphi$ of $\varphi$ along $f$
is the function from $T$ to $\Ze$ defined by 
$$f_*\varphi (y)=\int_{f^{-1}(y)}\varphi\, d\chi.$$
It is known that $f_*\varphi\in \F(T)$ and that $f_*:\F(S)\rightarrow
\F(T)$ is a morphism of additive groups.

Let $X\subset \R^n$ be a real algebraic set.
We say that a constructible function $\varphi$ on $X$ is algebraically
constructible if it can be written as a finite sum 
$$\varphi=\sum_{i\in I} m_i f_{i*}(\un_{X_i})$$
where $f_i$ are regular maps from real algebraic sets $X_i$ to $X$.
Algebraically constructible functions on $X$ form a subring, denoted by
$\A(X)$, of $\F(X)$. We say that a constructible function $\varphi$ on $X$ is
strongly algebraically
constructible if it can be written as a finite sum 
$$\varphi=\sum_{i\in I} m_i \un_{X_i}$$
where $X_i$ are real algebraic subsets of $X$.
Strongly algebraically constructible functions on $X$ form a subring
of $\A(X)$ denoted by
$\AS(X)$.

Let $A$ be a ring of semi-algebraic functions on $X$. 
For $f\in A$, we define the sign function associated to $f$
as $$\sign(f):X\rightarrow \{-1,0,1\}$$
$$x\mapsto \sign(f)(x)=\begin{cases} -1\,\,{\rm if}\,\,f(x)<0\\
0\,\,{\rm if}\,\,f(x)=0\\
1\,\,{\rm if}\,\,f(x)>0\\ \end{cases}$$
Let $f\in A$, we have
$\sign(f)\in\F(X)$ since $f$ is a semi-algebraic function (the inverse image of a semi-algebraic
set by a semi-algebraic map is a semi-algebraic set \cite[Prop. 2.2.7]{BCR1}).
Following \cite{ABR}, we say that two $n$-tuples $<f_1,\ldots,f_n>$ and $<h_1,\ldots,h_n>$ of elements
of $A$ are
equivalent, and we write $<f_1,\ldots,f_n>\simeq <h_1,\ldots,h_n>$, if 
$$\forall x\in X,\,\,\sign(f_1(x))+\cdots+\sign(f_n(x))=\sign(h_1(x))+\cdots+\sign(h_n(x)).$$
A (quadratic) form over $A$ is
an equivalence class of a $n$-tuple for this relation. If $\rho$ is the class of
the $n$-tuple $<f_1,\ldots,f_n>$, we simply write
$\rho=<f_1,\ldots,f_n>$ and $n$ is called the dimension of $\rho$ and
denoted by $\dim(\rho)$.
For two forms $<f_1,\ldots,f_n>$ and $<g_1,\ldots,g_m>$ over $A$, we
define the sum  (denoted by $\perp$) and the product (denoted by  $\otimes$):
$$<f_1,\ldots,f_n>\perp
<g_1,\ldots,g_m>=<f_1,\ldots,f_n,g_1,\ldots,g_m>,$$
$$<f_1,\ldots,f_n>\otimes
<g_1,\ldots,g_m>=<f_1g_1,\ldots,f_ng_1,f_1g_2,\ldots,f_ng_2,\ldots,f_ng_m>.$$
We call two forms $<f_1,\ldots,f_n>$ and $<g_1,\ldots,g_m>$ over $A$
similar, and write $$<f_1,\ldots,f_n>\sim <g_1,\ldots,g_m>,$$ if 
$$\forall x\in
X,\,\,\sign(f_1(x))+\cdots+\sign(f_n(x))=\sign(g_1(x))+\cdots+\sign(g_m(x)).$$
With the operations $\perp$ and $\otimes$, the set of similarity
classes of forms is a ring called the reduced Witt ring of degenerate
forms over $A$, we will denote it by $\Witt(A)$.
The form $\rho$ is called isotropic if there is a form $\tau$ with
$\rho\sim\tau$ and $\dim(\rho)>\dim(\tau)$. Otherwise, $\rho$ is
called anisotropic. The form $<0>$ is considered isotropic.

Since $A$ is a ring of semi-algebraic functions on $X$, we have a signature map $$\Lambda: \Witt(A)\rightarrow
F(X)$$ $$<f_1,\ldots,f_n>\mapsto \,\,
\sign(f_1)+\cdots+\sign(f_n)$$
which is a ring morphism. The signature map is clearly injective by
definition of similarity for forms.

Parusi\'nski and Szafraniec haved proved that algebraically constructible
functions correspond to sums of signs of polynomial functions.

\begin{thm} \cite[Thm. 6.1]{PS}\\
\label{ParSza}
Let $X\subset\R^n$ be a real algebraic set. Then 
$$\A(X)=\Lambda(\Witt(\Pol (X)))=\Lambda (\Witt(\SO (X))).$$
\end{thm}

We prove now that algebraically constructible
functions correspond to sums of signs of regulous functions. It is a
very natural result since the topology generated by zero sets of
regulous functions is the algebraically constructible topology. The
following theorem corresponds to Theorem A of the introduction.

\begin{thm}
\label{algconstregulu}
Let $X\subset\R^n$ be a real algebraic set. Then 
$$\A(X)=\Lambda(\Witt(\SR^0 (X))).$$
\end{thm}

\begin{proof}
We proceed by induction on the dimension of $X$. If $\dim (X)=0$ then
regulous means regular and the result follows from Theorem \ref{ParSza}.

Assume $\dim (X)>0$ and let $f\in\SR^0(X)$. Let $W$ denote $\pol(f)$.
There exist $p,q\in
\Pol(X)$ such that $f=\dfrac{p}{q}$ on $\dom(f)$ and $\Z(q)=W$. Notice
that $\Lambda(<f>)=\Lambda(<pq>)$ on $X\setminus W$.
We
have $f|_W\in \SR^0(W)$ and by induction there exists
$h_1,\ldots,h_k\in\Pol(W)$ such that $\Lambda
(<f|_W>)=\Lambda(<h_1,\ldots,h_k>)$. The polynomial functions $h_i$
are restrictions of polynomial functions on $X$ still denoted by $h_i$
\cite[prop. 3.2.3]{BCR}. 
The proof is done since 
$$\Lambda(<f>)=\Lambda(<pq>\perp <1,-q^2>\otimes <h_1,\ldots,h_k>)$$
on $X$.
\end{proof}

In the next section, we will count the number of signs of polynomial
functions we need in the sum to be the sign of a regulous function.

We prove now that strongly algebraically constructible functions are
exactly finite sums of characteristic functions of regulous closed sets.
\begin{prop}
\label{fortalgconst}
Let $X\subset\R^n$ be a real algebraic set. Then
$$\AS(X)=\{\,\,\sum_{i\in I}m_i\un_{W_i}\,,\,I\,{\rm
  finite},\,m_i\in\Ze,\,W_i\subset X\,{\rm regulous\, closed}\}.$$
\end{prop}

\begin{proof}
Let $W$ be a closed regulous subset of $X$. Let $f\in\SR^0(X)$ be such
that $\Z(f)=W$.
By \cite[Thm. 4.1]{FHMM} and since $f$ is the restriction to $X$ of a
regulous function on $\R^n$, 
there exists a finite stratification $X=\coprod_{i\in I}W_i$ with $W_i$
Zariski locally closed subsets of $X$ such that $f|_{W_i}$ is regular $\forall
i\in I$. Given $i\in I$, there are $p_i,q_i\in\Pol (X)$ such
that $\dfrac{p_i}{q_i}|_{W_i}=f|_{W_i}$ and $\Z(q_i)\cap
W_i=\emptyset$. Hence $S_i=W\cap W_i=\Z(p_i)\cap W_i$ is
also Zariski locally closed.
So there is a finite stratification
$W=\coprod_{i\in I} S_i$ with
$S_i$ Zariski locally closed subsets of $X$. It means that
$S_i=Z_i\cap (X\setminus Z'_i)$ where $Z_i$ and $Z'_i$ are real algebraic
subsets of $X$. Then $$\un_W=\sum_{i\in I}\un_{S_i}=\sum_{i\in
  I}(\un_{Z_i}(\un_X-\un_{Z'_i}))=\sum_{i\in I}(\un_{Z_i}-\un_{Z_i\cap
  Z'_i})\in \AS(X).$$
\end{proof}

We characterize algebraically constructible functions using regulous
closed sets and regulous maps.

Let $W\subset\R^n$ be a regulous closed set. A map $W\to \R^m$ is
called regulous if its coordinate functions are regulous on $W$ i.e
are restrictions to $W$ of regulous functions on $\R^n$ (see \cite{FHMM}).

\begin{thm}
\label{algconst}
Let $X\subset\R^n$ be a real algebraic set. Then
$$\A(X)=\{\,\,\sum_{i\in I}m_if_{i*}(\un_{W_i})\,,\,I\,{\rm
  finite},\,m_i\in\Ze,\,W_i\,{\rm regulous\, closed},
\,f_i:W_i\rightarrow X\,{\rm regulous\, map}\}.$$
\end{thm}

\begin{proof}
By Proposition \ref{fortalgconst} and since $f_*$ is additive, it is
sufficient to prove that $f_*(\un_Y)\in \A(X)$ when $f:Y\rightarrow X$
is a regulous map between two real algebraic sets. We proceed by induction on
the dimension of $Y$. If $\dim (Y)=0$ then $f$ is regular and there is
nothing to prove. 
Assume $\dim(Y)>0$. We may also assume that $Y$ is irreducible. By
\cite[Thm. 3.11]{FHMM}, there exists a proper regular birational map $\pi:
\tilde{Y}\rightarrow Y$ such that $f\circ \pi $ is a regular map (solve the singularities of $Y$ and then use
\cite[Thm. 3.11]{FHMM}). The birational map $\pi$ is biregular from
$\tilde{Y}\setminus \pi^{-1}(Z)$ to $Y\setminus Z$ with $Z$ a real
algebraic subset of $Y$ of positive codimension. Then 
$$f_*(\un_Y)=(f\circ\pi)_*(\un_{\tilde{Y}})-(f\circ\pi)_*(\un_{\pi^{-1}(Z)})+f_*(\un_Z)$$
and $f_*(\un_Z)\in\A(X)$ by the induction hypothesis.
\end{proof}

Now we look at sum of signs of rational continuous functions. Before
that, we recall the definition of Nash contructible functions
introduced by C. McCrory and A. Parusi\'nski (see
\cite{MP1}).

Let $X\subset \R^n$ be a real algebraic set.
We say that a constructible function $\varphi$ on $X$ is Nash 
constructible if it can be written as a finite sum 
$$\varphi=\sum_{i\in I} m_i f_{i*}(\un_{T_i})$$
where $f_i$ are regular and proper maps from real algebraic sets $X_i$
to $X$ and $T_i$ is a connected component of $X_i$.
Nash constructible functions on $X$ form a subring, denoted by
$\NA(X)$, of $\F(X)$. Clearly, $$\A(X)\subset \NA(X).$$ The characteristic
function of a connected component of a smooth irreducible real algebraic curve
with $2$ connected components is a Nash contructible function that is
not algebraically constructible.

In \cite{Bo3}, I. Bonnard has studied the connection between Nash
constructible functions and sum of signs of semi-algebraic
arc-analytic functions.
\begin{thm} \cite[Prop. 4, Thm. 2]{Bo3}
\label{bonnard}
Let $X\subset \R^n$ be a real algebraic set. A sum of signs of
semi-algebraic arc-analytic functions on $X$ is Nash
constructible. The converse is true if $X$ is compact.
\end{thm}

Even if rational continuous functions on a central real algebraic set
are semi-algebraic but not necessarily arc-analytic (see Example
\ref{ratcontnotarcana}), we may wonder if their signs are Nash
constructible functions.
In the following example, we prove that the sign of a rational
continuous function on a central algebraic set is not always an
algebraically constructible function nor a Nash constructible function.

\begin{ex}
\label{referee}
{\rm Consider the real algebraic set $S=\Z
  ((x^6+y^2+zx^4)((x+zy)^6+(-xz+y)^2-z(x+zy)^4))\subset \R^3$. In
  fact, $S=V\cup \Phi^{-1}(W)$ where $V=\Z(x^6+y^2+zx^4)$,
  $W=\Z(X^6+Y^2-ZX^4)$ and $\Phi:\R^3\rightarrow\R^3$ is the biregular
  map given by $\Phi(x,y,z)=(x+zy,-xz+y,z)=(X,Y,Z)$. The singular locus of $S$
  is the $z$-axis. The non-negative part of the $z$-axis is included
  in $\overline{\Phi^{-1}(W)_{reg}}^{eucl}$ and its complement is
    contained in $\overline{V_{reg}}^{eucl}$. Consequently, $S$ is
    central. We consider the rational function $f=\dfrac{y}{x}$ on
    $S$. The function $f$ is regular outside the $z$-axis and the
    Zariski closed set $A=\Z(z^6y^4+1-z^5y^2)\cap S$. Notice that $A$ does not
    meet the $z$-axis and that $z^6y^4+1-z^5y^2$ is constant to $1$ on
    the $z$-axis. It is not difficult to see that we can extend
    continuously $f$ to the negative part of the $z$-axis by a
    function identically zero. To understand what happens for $f$ on the positive part of the
    $z$-axis, we have to use the biregular map $\Phi$. Remark that
    the image by $\Phi$ of the $z$-axis is the $Z$-axis and more
    precizely $\Phi(0,0,z)=(0,0,z)=(0,0,Z)$. We have
    $x=\dfrac{X-ZY}{1+Z^2}$, $y=\dfrac{Y+ZX}{1+Z^2}$ and
    $\dfrac{y}{x}=\dfrac{Y+ZX}{X-ZY}=\dfrac{Y/X+Z}{1-ZY/X}$. Since
    $\dfrac{Y}{X}$ seen on $W$ can be extended to the positive part of the
    $Z$-axis by a function identically zero then $f$ can be extended
    continuously to the positive part of the $z$-axis by a function
    equal to $z$. By multiplying by a sufficiently big power
    of $z^6y^4+1-z^5y^2$ (see \cite[Prop. 2.6.4]{BCR} or
    \cite[Lem. 5.1]{FHMM}), we get a rational continuous function on
    $S$ again denoted by $f$. A referee of an earlier version of the
    paper has given this example in order to prove that zero sets of
    rational continuous functions on central algebraic sets are not
    always regulous closed. Indeed, assume that $\Z(f)$ is regulous
    closed then it is also the zero set of a regulous function on
    $S$. Since restrictions of regulous functions to the $z$-axis are
    regulous 
    and thus regular \cite[Prop. 2.4]{FMQ} then
    the intersection of $\Z(f)$ with the $z$-axis (equal to half of
    the $z$-axis) is Zariski closed, a contradiction. We can derive
    many other consequences from this example:\\
$\bullet$ $\Z(f)$ is not an arc-symmetric set (see \cite{KK}).\\
$\bullet$ $\sign(f)$ is not an algebraically constructible function:
Assume $\sign(f)\in\A(S)$. So $\sign(f)$ is a sum of signs of
polynomial functions. Since the restriction of a polynomial function
to a real algebraic subset is a polynomial function then it follows
that the restriction of $\sign(f)$ to the $z$-axis is an algebraically
constructible function. We get a contradiction because 
the algebraically constructible functions on the $z$-axis are the
constructible functions that are generically constant mod $2$
\cite[Ex. 2.3]{MP1}.\\
$\bullet$ $\sign(f)$ is not a Nash constructible function: By
\cite{Bo3} the restriction of a Nash constructible function to the
$z$-axis (which is an irreducible arc-symmetric set) must be generically
constant mod $2$.}
\end{ex}

\begin{rem}
To conclude this section, we remark that it follows from above
results that, if $X\subset \R^n$ is a real algebraic set, the
following rings $\Witt(\Pol(X))$, $\Witt(\SO(X))$, $\Witt(\SR^0(X))$,
$\A(X)$ are all isomorphic. 
\end{rem}

\section{Lengths of signs of regulous functions (part 1)}

Throughout this section $X$ will denote a real algebraic subset of
dimension $d$ of $\R^n$. By Theorem \ref{algconst}, the sign of a
regulous function on $X$ can be written as a sum of signs of
polynomial functions on $X$. The goal of this section is to bound in
terms of $d$ the number of polynomial functions
needed in such representation. This is connected to the work of I. Bonnard 
(\cite{Bo1} and \cite{Bo2}) that concerns the representation of general
algebraically constructible functions as sums of signs of polynomial
functions. However, the author cautions the reader that most of the results of this text concern specifically algebraically
constructible functions that are signs of regulous functions and
depend strongly of the nice properties verified by the regulous
functions. It seems unlikely to be able to generalize the results obtained for
the signs of regulous functions to general algebraically constructible
functions.

\subsection{Length of an algebraically constructible function}

\begin{defn}
$\bullet$ Given $\varphi\in \A(X)$, the number $\length(\varphi)$, called the length of $\varphi$,
will denote the smallest
integer $l$ such that
$\varphi$ can be written as a sum of $l$ signs of polynomial
functions on $X$. So there is a form $\rho$ over $\Pol(X)$ such that
$\Lambda(\rho)=\varphi$ on $X$ and $\dim(\rho)=\length(\varphi)$. It is
clear that $\rho$ is anisotropic and then it is unique. We denote by
$\rho(\varphi)$ the corresponding anisotropic form of dimension
$\length(\varphi)$.\\
$\bullet$ Let $f$ be a semi-algebraic function on $X$ such that
$\sign(f)\in \A(X)$. We simply denote by $\length(f)$ the length of
$\sign(f)$, it is called the length of the
sign of $f$. We also denote by $\rho(f)$ the form $\rho(\sign(f))$.\\
$\bullet$ 
Let $B$ be a ring of semi-algebraic functions on $X$ such that $\Lambda(\Witt(B))\subset\A(X)$.
The length of $B$ is the smallest number $\length(B)=l\geq 1$ such
that any $f\in B$ has $\length(f)\leq l$, and $\length(B)=+\infty$ if
such integer does not exist. 
\end{defn}

In the following, the goal is to prove that $\length(\SR^0(X))$ is finite and
to give upper bounds for $\length(\SR^0(X))$ in terms of the dimension
$d$ of $X$.

\begin{rem}
We clearly have $\length(\Pol(X))=\length(\SO(X))=1$. If $d=0$ then
regulous means regular and thus $\length(\SR^0(X))=1$.
\end{rem}

Let $f_1,\ldots, f_m$ be continuous semi-algebraic functions on $X$. In the sequel, we will use the following notations:
$$\S(f_1,\ldots,f_m)=\{x\in X|\,f_1(x)>0,\ldots,f_m(x)>0\}$$
$$\bar{\S}(f_1,\ldots,f_m)=\{x\in X|\,f_1(x)\geq 0,\ldots,f_m(x)\geq
0\}.$$
If all the functions $f_i$ lie in a ring $A$ of continuous semi-algebraic
functions, the set $\S(f_1,\ldots,f_m)$
(resp. $\bar{\S}(f_1,\ldots,f_m)$) is called $A$-basic open
(resp. $A$-basic closed). If $m=1$, we replace ``basic'' by
``principal''. If $A=\Pol (X)$ then we omet $A$. If $A=\SR^0(X)$, we
will sometimes write ``regulous basic'' (resp. ``regulous principal'')
instead of ``$\SR^0(X)$-basic'' (resp. ``$\SR^0(X)$-principal'').

In the following example, we prove that even for curves the sign of a
regulous function is not always the sign of a polynomial function.
\begin{ex}
{\rm Let $X=\Z (y^2-x^2(x-1))$ considered in Example \ref{exemple1} and let $f$ be the restriction to $X$ of the plane 
regulous function $1-\dfrac{x^3}{x^2+y^2}$. The function $f$ is zero
on the one-dimensional connected component of $X$ and has value $1$ at
the isolated point of $X$. If a polynomial function $g$ has the
sign of $f$ on the one-dimensional connected component of $X$ then $g$
vanishes on whole $X$ since $X$ is Zariski irreducible.
However the sign of $f$ is the sum of signs of two polynomial
functions on $X$, more precisely we have $\rho(f)=<1,-(x^2+y^2)>$ and
therefore $\ell(f)=2$.}
\end{ex}

\subsection{The polar depth of a regulous function}

We give upper bounds on $\length(\SR^0(X))$ introducing the polar
depth of a regulous function on $X$.

\begin{defn}
Let $f\in \SR^0(X)$.\\
We set $f_0=f$, $X_0=X$ and $X_1=\pol(f_0)$.\\
If $X_1\not=\emptyset$ i.e if $f_0$ is not regular on $X_0$ then we
set $f_1=f_0|_{X_1}\in\SR^0(X_1)$ and $X_2=\pol(f_1)$.\\
By repeating the same process, it stops after at most $d$ steps since
$\dim(X_{i+1})<\dim (X_i)$ and $X_{i+1}=\emptyset$ if $\dim X_i=0$.\\
At the step of index $i$ we associate to the regulous function $f_i$
on $X_i$ a rational representation $(p_i,q_i)\in \Pol(X)\times\Pol(X)$
such that $f_i=\dfrac{p_i}{q_i}$ on $X_i\setminus X_{i+1}$ and
$\Z(q_i)\cap X_i=X_{i+1}$.\\
The following sequence
$$((f_0,X_0,p_0,q_0),\ldots,(f_k,X_k,p_k,q_k))$$ 
is called a ``polar sequence'' associated to $f$.
We have $X_i\not=\emptyset$ for $i=1,\ldots,k$ and $X_{k+1}=\emptyset$
i.e $f_k$ is regular on $X_k$.\\
The number $k$ of the previous sequence is called the ``polar depth''
of $f$ and we denote it by $\poldepth(f)$.
\end{defn}

\begin{rem}
If $f\in \SR^0(X)$ then obviously $\poldepth(f)\leq d$.
\end{rem}

\begin{prop}
\label{nbsignpol}
Let $X\subset\R^n$ be a real algebraic set of dimension $d$. Let $f\in
\SR^0(X)$, $k=\poldepth(f)$ and $((f_0,X_0,p_0,q_0),\ldots,(f_k,X_k,p_k,q_k))$
a ``polar sequence'' associated to $f$. Then 
$$\Lambda(<f>)=\Lambda(<p_0q_0>\perp_{i=1}^k  (<1,-\prod_{j=0}^{i-1}q_{j}^2>\otimes
<p_iq_i>))$$ on $X$.
Therefore, $$\length(f)\leq 1+2\poldepth(f).$$
\end{prop}

\begin{proof}
The proof is straightforward since
we have $\Lambda(<f>)=\Lambda(<p_0q_0>)$ on $X\setminus X_1$ and 
$$\Lambda(<f>)=\Lambda(<p_0q_0>\perp_{i=1}^m  (<1,-\prod_{j=0}^{i-1}q_{j}^2>\otimes
<p_iq_i>))$$
on $X\setminus X_{m+1}$ for $m=1,\ldots,k$ and $X_{k+1}=\emptyset$.
\end{proof}

It follows from Propositions \ref{nbsignpol}:
\begin{thm}
\label{nbsigngen}
Let $X\subset\R^n$ be a real algebraic set of dimension $d$. Then
$$\length(\SR^0(X))=1\,\,{\rm if}\,d=0,$$
$$\length(\SR^0(X))\leq 2d+1\,\,{\rm else}.$$
\end{thm}

\subsection{Continuous semi-algebraic functions with length of sign
  equal to one}

We will use several times the following lemma which is a consequence
of \L ojasiewicz inequality.
\begin{lem} \cite[Lem. 7.7.10]{BCR1}\\
\label{hl1}
Let $S$ be a closed semi-algebraic subset of $X$. Let $f,g\in\Pol
(X)$. There exist $p,q\in\Pol (X)$ such that $p>0$ on $X$, $q\geq 0$
on $X$, $\Lambda(<pf+qg>)=\Lambda(<f>)$ on $S$ and $\Z
(q)=\overline{\Z(f)\cap S}^{Zar}$.
\end{lem} 

The following theorem provides a characterization of the signs of
continuous semi-algebraic functions that are algebraically
constructible of length equal to one.
\begin{thm}
\label{signsa}
Let $f$ be a continuous
semi-algebraic function on $X$. There exists $p\in\Pol(X)$ such that 
$\Lambda(<f>)=\Lambda(<p>)$ (i.e $\sign(f)\in\A (X)$ and
$\length(f)\leq 1$) if and only if the three following
conditions are satisfied:\\
1) $\Z(f)$ is Zariski closed.\\
2) $\S(f)$ is principal.\\
3) $\S(-f)$ is principal or equivalently $\bar{\S}(f)$ is principal.
\end{thm}

\begin{proof}
One implication is trivial. For the other one, assume there exist
three polynomial functions $p_1,p_2,p_3$ on $X$ such that
$\S(f)=\S(p_1)$, $\S(-f)=\S(-p_2)$ and $\Z(f)=\Z(p_3)$. We may replace
$p_1$ by $p_1p_3^2$ and we get:
$$\S(f)=\S(p_1)\,\,{\rm and}\,\,\Z(f)\subset\Z(p_1).$$
Let $S$ denote the closed semi-algebraic set
$\bar{\S}(f)$. By Lemma \ref{hl1}, there exist $p,q\in\Pol (X)$ such that $p>0$ on $X$, $q\geq 0$
on $X$, $\Lambda(<pp_1+qp_2>)=\Lambda(<p_1>)$ on $S$ and $\Z
(q)=\overline{\Z(p_1)\cap S}^{Zar}$. Let $h$ denote the polynomial
function $pp_1+qp_2$. We want to prove that $\Lambda(<h>)=\Lambda(<f>)$ on
$X$. We have
$\Lambda(<h>)=\Lambda(<p_1>)=\Lambda(<f>)$ on $S$ since $\S(f)=\S(p_1)$
and since $\Z(f)\subset \Z(p_1)$.
Assume now $x\not\in S$. Notice that it is equivalent to suppose that $f(x)<0$. So $p_2(x)<0$
(since $\S(-f)=\S(-p_2)$), $p_1(x)\leq 0$ (since $\bar{\S}(-f)=\bar{\S}(-p_1)$). 
The proof is done if we prove that $q(x)>0$ since in that case we
would have $h(x)<0$.
We have $S\cap \Z(p_1)=\bar{\S}(f)\cap \Z(p_1)\subset \Z(f)\cap
\Z(p_1)=\Z(f)$ since $\S(f)=\S(p_1)$ (you can not have simultaneously $f(y)>0$
and $p_1(y)=0$). Since $\Z(f)$ is Zariski
closed, we get
$$\Z(q)=\overline{\Z(p_1)\cap S}^{Zar}\subset \overline{\Z(f)}^{Zar}=\Z(f)$$
and it follows that $x\not\in \Z(q)$.
\end{proof}

\begin{rem} 
\label{Z-closed1} Look at Theorem \ref{Z-closed2} for an improvement
  of Theorem \ref{signsa} in the case the continuous semi-algebraic function $f$ is regulous.
\end{rem}

\subsection{The case of curves}

If $X$ is a curve then we know by Theorem \ref{nbsigngen} that
$\length(\SR^0(X))\leq 3$. We improve the upper bound when $X$
satisfies several different hypotheses.

We give a one dimensional version of Theorem \ref{signsa}.
\begin{prop}
\label{signsadim1}
Assume $\dim(X)=1$. Let $f$ be a continuous
semi-algebraic function on $X$. There exists $p\in\Pol(X)$ such that 
$\Lambda(<f>)=\Lambda(<p>)$ if and only if $\Z(f)$ is Zariski closed.
\end{prop}

\begin{proof}
By \cite{Br}, any open semi-algebraic subset of $X$
is principal and thus $\S(f)$ and $\S(-f)$ are principal. We conclude
using Theorem \ref{signsa}.
\end{proof}

\begin{cor}
\label{signdim1lisse}
Assume $\dim(X)= 1$ and $X$ is smooth. Then $\length(\SR^0(X))=1.$
\end{cor}

\begin{proof}
In the case $X$ is a smooth real algebraic curve then the zero set of
a regulous function on $X$ is
Zariski closed since regulous means regular (see
\cite{FHMM}). The proof is done using Proposition \ref{signsadim1}.
\end{proof}

\begin{prop}
\label{signdim1}
Assume $\dim(X)= 1$, $X$ is central and irreducible. Then
$\length(\SR^0(X))=1.$
\end{prop}

\begin{proof}
Since $X$ is central and irreducible then $X$ is $\Co$-irreducible
(see \cite{FHMM}). Let $0\not=f\in\SR^0(X)$. If $\dim\Z(f)=1$ then $\Z(f)=X$. It follows that $\Z(f)$ is Zariski closed. The proof
follows now from Proposition \ref{signsadim1}.
\end{proof}

\begin{ex}
Let $X=\Z(x^2-y^3)\subset\R^2$ be the cuspidal curve and let
$f=\dfrac{y^2}{x}|_X$. We have $f\in\SR^0(X)\setminus\Pol(X)$ but 
$\Lambda(<f>)=\Lambda(<x>)$ on $X$.
\end{ex}

\begin{ex}
\label{exdim1red}
Let $X=\Z((y^2-x^2(x-1))(x+y^2))\subset\R^2$. It is not difficult to
see that $X$ is central and that the restriction $f$ to $X$ of the plane 
regulous function $1-\dfrac{x^3}{x^2+y^2}$ has a zero set that is not
Zariski closed. By Theorem \ref{signsa}, it follows that
$\length(f)\geq 2$. By this example, we prove that the hypothesis that
$X$ is irreducible is necessary in order to get the conclusion of
Proposition \ref{signdim1}. Remark that $\sign(f)=\Lambda(<x+y^2,
-x+\dfrac{1}{2}>$ i.e $\length(f)=2$.
\end{ex}

\begin{prop}
\label{signdim=1}
Assume $\dim(X)=1$ and $X$ is irreducible. Let $f\in\SR^0(X)$. There exist $h_1,h_2\in\Pol
(X)$ such that $\Lambda(<h_1,h_2>)=\Lambda(<f>)$ on $X$ i.e
$$\length(f)\leq \length(\SR^0(X))\leq 2.$$
\end{prop}

\begin{proof}
By the previous results we may assume that $\Z (f)$
is not Zariski closed. 
By \cite{FHMM}, $X=F\coprod \{x_1,\ldots,x_m\}$
where $F=\overline{X_{reg}}^{eucl}$ is the one-dimensional irreducible regulous
component of $X$ and $x_1,\ldots,x_m$ are the isolated points of $X$. 
Since $\Z(f)$ is not Zariski closed, we must have $\dim \Z(f)=1$ and
since $X$ is irreducible we get
$F\subset \Z(f)$ (see \cite{FHMM}).
For each $x_i$ let
$p_i\in\Pol(X)$ such that $p_i\geq 0$ on $X$ and $\Z(p_i)=\{x_i\}$.
We set $h_1$ to be 
the product of the $p_i$ such that $f(x_i)\leq 0$ and $h_2$ to be
the $(-1)\times$ the product of the $p_i$ such that $f(x_i)\geq 0$.
For this choice of $h_1$ and $h_2$, we get the proof.
\end{proof}

\begin{rem}
The previous proof works also in the reducible case if we know
that $\dim (\Z(f)\cap Y)=1$ for any irreducible component of dimension
one $Y$ of $X$. 
Let $f$ be a regulous function on a reducible real algebraic curve
$X$. Assume now we have two irreducible components of
dimension one $Y_1,Y_2$ of $X$ such that $\dim(\Z(f)\cap Y_1)=1$ and
$\dim(\Z(f)\cap Y_2)=0$. By Propositions \ref{signdim=1} and \ref{signsadim1}, there exist $h,h_1,h_2\in\Pol
(X)$ such that $\Lambda(<h_1,h_2>)=\Lambda(<f>)$ on $Y_1$ and such
that $\Lambda(<h>)=\Lambda(<f>)$ on $Y_2$, and it is not clear if we
can patch together these two representations to get a representation
of $\sign(f)$ on $Y_1\cup Y_2$ as we have done in Example
\ref{exdim1red}.
\end{rem}

\subsection{Upper bounds on the length of the
  ring of regulous functions on normal real algebraic sets}

The polar locus of a regulous function on $\R^n$ has codimension $\geq
2$ \cite[Prop. 3.5]{FHMM}. We generalize this result in the following proposition.
\begin{prop}
\label{normal}
If $f\in\SR^0(X)$ then 
$\codim ((\pol(f)\cap Y)\setminus \Sing(Y))\geq 2$ for any irreducible
component $Y$ of $X$.
\end{prop}

\begin{proof}
We may assume $X$ is irreducible and suppose $\dim ((\pol(f)\setminus
\Sing(X))=d-1$. Under this assumption there exists a resolution of singularities
$\pi: \tilde X\to X$ of $X$ and also of $\pol(f)$
such that $\tilde f=f\circ \pi\in\SR^0(\tilde X)$, $\pol(\tilde f)=Z$ where $Z$ is the strict transform of
$\pol(f)$ and $\dim Z=d-1$. Let $W$ be an irreducible component of $Z$
of dimension $d-1$. Since the local ring $\SO_{\tilde X,W}$ is a
discrete valuation ring, we may write the rational function $\tilde
f=t^mu$ with $t$ an uniformizing parameter of $\SO_{\tilde X,W}$, $m<0$
and $u$ a unit of $\SO_{\tilde X,W}$. There exists a non-empty Zariski open
subset $U$ of $W$ where $u$ doesn't vanish and thus it is impossible
to extend continuously the rational function $t^mu$ to $W$, a
contradiction.
\end{proof}

\begin{cor}
\label{cornormal}
Let $X\subset\R^n$ be a real algebraic set of dimension $d\geq 1$ such that
$\codim(\Sing(Y))>1$ for any irreducible component $Y$ of $X$. Let
$f\in\SR^0(X)$ then $$\codim (\pol(f))> 1$$ and $$\poldepth(f)\leq d-1.$$
\end{cor}

It follows from Corollaries \ref{signdim1lisse}, \ref{cornormal}
and Proposition \ref{nbsignpol}:
\begin{thm}
\label{nbsignnormal}
Let $X\subset\R^n$ be a real algebraic set of dimension $d$ such that
$\codim(\Sing(Y))>1$ for any irreducible component $Y$ of $X$. Then
$$\length(\SR^0(X))=1\,\,{\rm if}\,d=0\,{\rm or}\,1,$$
$$\length(\SR^0(X))\leq 2d-1\,\,{\rm else}.$$
\end{thm}

\begin{rem}
Recall that an irreducible real algebraic set $Y\subset\R^n$ is called
normal if its ring of polynomial functions $\Pol(Y)$ is integrally
closed in $\K(X)$. It is well known that if $Y$ is normal then
$\codim(\Sing(Y))>1$. Therefore, the previous theorem applies when $X$
is real algebraic set with normal irreducible components. It also
applies when $\codim(\Sing(X))>1$.
\end{rem}

We will improve the results of Theorems
\ref{nbsignnormal} and \ref{nbsigngen} in the sixth section.

\begin{ex}
{\rm We prove the optimality of the bound given in Theorem
\ref{nbsignnormal} for $X=\R^2$ and thus for $d=2$ i.e we show that $\length(\SR^0(\R^2))=3$. Consider the regulous function
$f=-1+\dfrac{x^3}{x^2+y^2}$. Notice that we have a partition of $\R^2$
given by $\R^2=\S(-f)\coprod
\Z(f)\coprod \S(f)$.
We can not write $\Lambda(<f>)=\Lambda(<h>)$ with $h\in \R[x,y]$ since
$\Z(f)$ is not Zariski closed.\\
We can not write $\Lambda(<f>)=\Lambda(<h_1,h_2>)$ with $h_1,h_2\in
\R[x,y]$ since it would imply that $h_1h_2$ vanishes on
$\S(-f)\cup \S(f)$ and thus vanishes on whole $\R^2$.\\
By Proposition \ref{nbsignpol}, we get 
$$\rho(f)=<-x^2-y^2+x^3,-1, x^2+y^2>.$$}
\end{ex}

\section{Regulous principal semi-algebraic sets}

\subsection{Regulous principal semi-algebraic sets versus polynomial 
  principal semi-algebraic sets}

Let $X\subset \R^n$ be a real algebraic set of dimension $d$. 

In this
section we raise and study the following questions:\\
Given a regulous principal open (resp. closed) semi-algebraic subset
of $X$, is it a principal
open (resp. closed) semi-algebraic subset of $X$?\\
By taking the complementary set, we only have to look at the question
concerning open sets.  If $d=0$ the answer is
trivially ``yes''. For $d=1$ the
answer is also ``yes'' by \cite{Br} since in this case any open
(resp. closed) semi-algebraic subset of $X$ is principal.

For $d=2$ the answer can be negative:
\begin{ex}
{\rm As usual consider $X=\R^2$ and $f=1-\dfrac{x^3}{x^2+y^2}$. Let
$S=\S(f)$. 
Since $S\cap \overline{\Bd(S)}^{Zar}=\{(0,0)\}\not=\emptyset$ then $S$
cannot be basic \cite[Prop. 2.2]{Br}
($\Bd(S)=\overline{S}^{eucl}\setminus \mathring{S}$).}
\end{ex}

In the following we will prove that under the topological condition ``$S\cap \overline{\Bd(S)}^{Zar}=\emptyset$'', the
answer to the previous question, for the regulous principal open
semi-algebraic set $S$, is ``yes''.

\begin{rem}
Let $f\in \SR^0(X)$. Set $S=\S(f)$ and assume $f=\dfrac{p}{q}$ on
$\dom(f)$ with $p,q\in\Pol(X)$ and $\Z(q)=\pol(f)$. If we assume in addition
that $S\cap \overline{\Bd(S)}^{Zar}=\emptyset$, we will prove later that there exists $r\in\Pol(X)$ such that $S=\S(r)$ but
it may happen that we can not choose $r$ to be equal to $pq$.
Consider $X=\R^2$,
$f=\dfrac{y^2+x^2(1-x)^2}{x^2+y^2}=\dfrac{p}{q}$. Since
$f=1+\dfrac{x^4-2x^3}{x^2+y^2}$ then we see that $f\in\SR^0(\R^2)$.
We have
$S=\S(f)=\R^2\setminus\{(1,0)\}$, $\overline{\Bd(S)}^{Zar}=\{(1,0)\}$
and $\S(pq)=\R^2\setminus\{(1,0),(0,0)\}$.
\end{rem}

We can answer affirmatively to the previous question if the set $S$
does not meet the polar locus.
\begin{prop}
\label{sanspoles}
Let $f\in\SR^0(X)$ and $S=\S(f)$. Assume $S\cap\pol(f)=\emptyset$. The
set $S$
is then a principal open semi-algebraic set and more precisely we have
$\S(f)=\S(pq)$ where $p,q\in\Pol(X)$ satisfy $f=\dfrac{p}{q}$ on
$\dom(f)$ and $\Z(q)=\pol(f)$.
\end{prop}

\begin{proof}
Assume $f=\dfrac{p}{q}$ on
$\dom(f)$ with $p,q\in\Pol(X)$ and $\Z(q)=\pol(f)$. We clearly have
$\S(f)\setminus\pol(f)=\S(pq)\setminus \pol(f)=\S(pq)$. By assumption
$\S(f)\setminus\pol(f)=\S(f)$ and thus $\S(f)=\S(pq)$.
\end{proof}

\begin{rem} 
Let $f\in \SR^0(X)$. Set $S=\S(f)$ and assume $((f_0,X_0,p_0,q_0),\ldots,(f_k,X_k,p_k,q_k))$ is
a polar sequence associated to $f$. We have 
$$S=\coprod_{i=0}^k \S(p_iq_i)\cap X_i.$$
\end{rem}

We will use several times the following other consequence of H\"ormander-\L
ojasiewicz inequality.
\begin{lem} \cite[Prop. 1.16, Chap. 2]{ABR}\\
\label{hl2}
Let $C$ be a closed semi-algebraic subset of $X$ and let
$f,g\in\Pol(X)$ such that $\Z(f)\cap C\subset \Z(g)$. There exist
$h\in\Pol(X)$ and $l\in\N$ odd such that 
$$\Lambda(<(1+h^2)f+g^l>)=\Lambda(<f>)$$ on $C$.
\end{lem}

The following theorem is the main result of the section. It implies
Theorem C of the introduction.
\begin{thm}
\label{principal1}
Let $f\in\SR^0(X)$ and $S=\S(f)$. There exists $r\in\Pol(X)$ such that 
$$ \S(r)\subset S\,\,{\rm and}\,\,S\setminus \S(r)\subset
\overline{\Bd(S)}^{Zar}\cap \pol(f).$$
More precisely, if $((f_0,X_0,p_0,q_0),\ldots,(f_k,X_k,p_k,q_k))$ is
a polar sequence associated to $f$ then, for $i=0,\ldots,k$, there
exists $r_i\in\Pol(X)$ such that 
$$ \S(r_i)\cap X_i\subset S\cap X_i\,\,{\rm
  and}\,\,(S\setminus \S(r_i))\cap X_i\subset
\overline{\Bd(S)}^{Zar}\cap X_{i+1}.$$
\end{thm}

\begin{proof} We set $S_i=S\cap X_i$ for $i=0,\ldots,k$.
We proceed by decreasing induction on $i=k,\ldots,0$.\\
$\bullet$ For $i=k$ there is nothing to do since $f_k$ is regular on $X_k$.\\
$\bullet$ Assume $i\in\{0,\ldots,k-1\}$ and there exists $r_{i+1}\in\Pol(X)$ such that 
$$ \S(r_{i+1})\cap X_{i+1}\subset S\cap X_{i+1}\,\,{\rm
  and}\,\,(S\setminus \S(r_{i+1}))\cap X_{i+1}\subset
\overline{\Bd(S)}^{Zar}\cap X_{i+2}.$$
Let $F$ denote the closed semi-algebraic subset of $X_i$ defined by 
$F=\overline{\S(r_{i+1})\cap X_i}^{eucl}\cap (X_i\setminus S_i).$

We have
\begin{equation}
\label{equ1}
X_{i+1}\cap F\subset \Z(r_{i+1})\cap X_i.
\end{equation}
If $x\in X_{i+1}\cap F$ then $x\in X_{i+1}$ and $x\not\in S_i\cap
X_{i+1}=S_{i+1}$. By induction hypothesis we have $\S(r_{i+1})\cap
X_{i+1}\subset S_{i+1}$ and thus $r_{i+1}(x)\leq 0$. Since $x\in
\overline{\S(r_{i+1})\cap X_i}^{eucl}$ then $x\in
\overline{\S(r_{i+1})\cap X_i}^{eucl}\setminus (\S(r_{i+1})\cap
X_{i})=\Bd(\S(r_{i+1})\cap
X_{i})$ i.e $r_{i+1}(x)=0$ and it proves (\ref{equ1}).\\
\\
By (\ref{equ1}) and since $X_{i+1}=\Z(-q_i^2)\cap X_i$ then Lemma \ref{hl2} provides us
$h'\in\Pol(X)$, $l'$ an odd positive integer such that
$r'_{i+1}=(1+h'^2)(-q_i^2)+r_{i+1}^{l'}$ verifies 
$\Lambda(<r'_{i+1}>)=\Lambda(<-q_i^2>)$ on $F$. Since
$\Lambda(<r'_{i+1}>)=\Lambda(<r_{i+1}>)$ on $X_{i+1}$ then $r'_{i+1}$
satisfies the same induction hypotheses than $r_{i+1}$ namely
\begin{equation}
\label{equ2}
 \S(r'_{i+1})\cap X_{i+1}\subset S_{i+1}
\end{equation}
and
\begin{equation}
\label{equ3}
(S_{i+1}\setminus \S(r'_{i+1}))\subset
\overline{\Bd(S)}^{Zar}\cap X_{i+2}.
\end{equation}
We claim that $r'_{i+1}$ satisfies the third
property 
\begin{equation}
\label{equ4}
\S(r'_{i+1})\cap X_i\subset S_i.
\end{equation}
If $x\in \S(r'_{i+1})\cap X_i$ then $r_{i+1}(x)$ must be $>0$ and if
$x\not\in S_i$ then $x\in F$ and the sign of $r'_{i+1}(x)$ is the sign
of $-q_i^2(x)$, which is impossible. We have proved (\ref{equ4}).

Set $C=\overline{S_i}^{eucl}\setminus (\S(r'_{i+1})\cap X_i).$
Let $t\in\Pol(X)$ such that $\Z(t)=\overline{\Z(p_i)\cap
  C}^{Zar}$. Since $f_i\in \SR^0(X_i)$ then $\Z(q_i)\cap
X_i\subset\Z(p_i)\cap X_i$ \cite[Prop. 3.5]{FHMM} and thus we get 
$\Z(p_iq_i)\cap C\subset \Z(t)\subset \Z(t^2r'_{i+1}).$
By Lemma \ref{hl2}, there exist $h\in\Pol(X)$ and $l$ an odd positive
integer such that $r_i=(1+h^2)p_iq_i+t^{2l}r_{i+1}'^l$ verifies 
$\Lambda(<r_i>)=\Lambda(<p_iq_i>)$ on $C$. We prove now that $r_i$
is the function we are looking for.

Assume $x\in X_i\setminus S_i$. If $x\in X_{i+1}$ then
$p_i(x)q_i(x)=0$, else $x\in X_i\setminus (S_i\cup X_{i+1})$ and the
sign of $p_i(x)q_i(x)$ is the sign of $f_i(x)$; thus $p_i(x)q_i(x)\leq
0$. By (\ref{equ4}) we get
$r'_{i+1}(x)\leq 0$ and thus $r_i(x)\leq 0$. We have proved that 
\begin{equation}
\label{equ5}
\S(r_i)\cap X_i\subset S_i.
\end{equation}

It remains to prove 
\begin{equation}
\label{equ6}
S_i\setminus (\S(r_i)\cap X_i)\subset \overline{\Bd(S)}^{Zar}\cap
X_{i+1}.
\end{equation}
Assume $x\in S_i\setminus X_{i+1}$. We have
$f_i(x)=\dfrac{p_i(x)}{q_i(x)}$ and thus $p_i(x)q_i(x)>0$. If
$r'_{i+1}(x)\geq 0$ then $r_i(x)>0$. If $r'_{i+1}(x)< 0$ then $x\in C$
and we know that the sign of $r_i(x)$ is the sign of $p_i(x)q_i(x)$.
We have proved that $S_i\setminus X_{i+1}\subset \S(r_i)\cap
(X_i\setminus X_{i+1})$
and by (\ref{equ5}) then 
$S_i\setminus (\S(r_i)\cap X_i)\subset X_{i+1}.$ So in order to get
(\ref{equ6}) we are left to prove 
\begin{equation}
\label{equ7}
S_{i+1}\setminus (\S(r_i)\cap X_{i+1})\subset  \overline{\Bd(S)}^{Zar}.
\end{equation}

We have $(\Z(p_i)\cap C)\setminus X_{i+1}\subset \Bd(S_i)$ since
$C\subset \overline{S_i}^{eucl}$ and $S_i\setminus
X_{i+1}=(\S(p_iq_i)\cap X_i)\setminus X_{i+1}$. By (\ref{equ3}),
(\ref{equ4}) and
since $\Z(q_i)\cap X_i=X_{i+1}\subseteq
\Z(p_i)\cap X_i$ we get $\Z(p_i)\cap C\cap X_{i+1}=C\cap X_{i+1}=
((\overline{S_i}^{eucl}\setminus S_i)\cup (S_i\setminus
\S(r'_{i+1})))\cap X_{i+1}\subset (\Bd(S_i)\cap X_{i+1})\cup
(\overline{\Bd(S)}^{Zar}\cap X_{i+2})\subset
\overline{\Bd(S)}^{Zar}.$
From the above it follows that 
\begin{equation}
\label{equ8}
\Z(t)\subset
\overline{\Bd(S)}^{Zar}.
\end{equation}
Since $r_i=t^{2l}r_{i+1}'^l$ on $X_{i+1}$ then $S_{i+1}\setminus (\S(r_i)\cap
X_{i+1})=(S_{i+1}\setminus \S(r_{i+1}'))\cup (\Z(t)\cap S\cap
X_{i+1})$. Combining (\ref{equ3}) and (\ref{equ8}) we get
(\ref{equ7}), and the proof is complete.


\end{proof}


Remark that Theorem \ref{principal1} implies the first part of Proposition \ref{sanspoles}.
Let us mention consequences of Theorem \ref{principal1}. The following
result corresponds to Theorem C of the introduction.

\begin{thm}
\label{principal2}
Let $f\in\SR^0(X)$ and $S=\S(f)$. Then $S$ is a principal open
semi-algebraic set if and only if $S\cap
\overline{\Bd(S)}^{Zar}=\emptyset$.
\end{thm}

\begin{thm}
\label{principal2closed}
Let $f\in\SR^0(X)$. Then $\bar{\S}(f)$ is a principal closed
semi-algebraic set if and only if $\S(-f)\cap
\overline{\Bd(\S(-f))}^{Zar}=\emptyset$.
\end{thm}

\begin{proof} 
It is easily seen that an open (resp. closed) semi-algebraic subset
$S$ of $X$ is principal open (resp. closed) if and only if $X\setminus
S$ is principal closed (resp. open). According to the above remark,
the proof follows from Theorem \ref{principal2}.
\end{proof}

\begin{cor}
\label{principal3}
Let $f\in\SR^0(X)$ such that $\Z(f)$ is Zariski closed. Then the sets
$\S(f)$, $\S(-f)$, $\bar{\S}(f)$ and $\bar{\S}(-f)$ are principal
semi-algebraic sets.
\end{cor}

\begin{proof}
Assume $\Z(f)$ is Zariski closed. Since $\Bd(\S(f))\subset \Z(f)$, we
get $\S(f)\cap
\overline{\Bd(\S(f))}^{Zar}\subset \S(f)\cap\overline{\Z(f)}^{Zar} =\S(f)\cap
\Z(f)=\emptyset$. From the same arguments, we get $\S(-f)\cap
\overline{\Bd(\S(-f))}^{Zar}=\emptyset$. By Theorems \ref{principal2} and
\ref{principal2closed} the proof is complete.
\end{proof}

\subsection{Characterization of regulous principal semi-algebraic sets}

Let $X\subset \R^n$ be a real algebraic set of dimension $d$.

In this section, we will give an answer to the following question:
Under which conditions an open semi-algebraic set is regulous
principal?

\begin{defn}
A semi-algebraic subset $S$ of $X$ is said to be generically
principal on $X$ if $S$ coincides with a principal open
semi-algebraic subset of $X$ outside a real algebraic subset of $X$ of
positive codimension.
\end{defn}

The next result is a regulous version of Lemma \ref{hl2}.
\begin{lem}
\label{hl3}
Let $C$ be a closed semi-algebraic subset of $X$ and let
$f,g\in\SR^0 (X)$ such that $\Z(f)\cap C\subset \Z(g)$. There exist
$h\in\Pol(X)$ and $l\in\N$ odd such that $h>0$ on $X$ and 
$$\Lambda(<hf+g^l>)=\Lambda(<f>)$$ on $C$.
\end{lem}

\begin{proof}
We can see $C$ as a closed semi-algebraic subset of $\R^n$ and $f,g\in
\SR^0(\R^n)$ by definition of regulous functions on $X$. By
\cite[Thm. 2.6.6]{BCR}, for a sufficiently big positive odd integer
$l$ the function $\dfrac{g^l}{f}$ is semi-algebraic and continuous on
$C$. By \cite[Thm. 2.6.2]{BCR}, $|\dfrac{g^l}{f}|$ is bounded on $C$
by a
polynomial function $h$ which is positive definite on $\R^n$. The proof is done
by restricting these functions to $X$.
\end{proof}

\begin{prop}
\label{recolle}
Let $S$ be a semi-algebraic subset of $X$. The set $S$ is regulous principal open if and only if we
have:\\
1) $S\cap\overline{\Bd(S)}^{\Co}=\emptyset$,\\
and there exists an algebraic subset $W$ of $X$ of positive codimension
such that:\\
2) there exists $p\in\Pol(X)$ such that $S\setminus W=\S(p)\setminus
W$,\\
3) there exists $g\in\SR^0(X)$ such that $S\cap W=\S(g)\cap W$.
\end{prop}

\begin{proof}
Assume $S=\S(f)$ with $f\in \SR^0(X)$ such that $f=\dfrac{p}{q}$ on
$\dom(f)$ with $p,q\in\Pol(X)$ and $\Z(q)=\pol(f)$. 
We have $S\cap\overline{\Bd(S)}^{\Co}=\emptyset$ since
$\overline{\Bd(S)}^{\Co}\subset \Z(f)$. Moreover $S\setminus
\pol(f)=\S(pq)$ and $f|_{\pol(f)}\in\SR^0(\pol(f))$. We have proved
one implication.

Assume now $S$ satisfies the the three conditions of the Proposition.
We may assume $W\subset \Z(p)$ changing $p$ by $q^2p$ where $q\in
\Pol(X)$ satisfies $W=\Z(q)$.

Set $F=\overline{\S(g)}^{eucl}\setminus S$. Assume $x\in W\cap F$ then
$x\in W\setminus (S\cap  W)$ and thus $g(x)\leq 0$. Then $x\in
\Bd(\S(g))\subset \Z(g)$. We have proved that $\Z(-q^2)\cap F\subset
\Z(g).$ By Lemma \ref{hl3} there exist
$h\in\Pol(X)$, $l\in\N$ odd and $g'\in\SR^0(X)$ such that $h>0$ on
$X$, $g'=-hq^2+g^l$ and 
$\Lambda(<g'>)=\Lambda(<-q^2>)$ on $F$. Clearly, the function $g'$ satisfies 
again the property 3) of the proposition, namely 
\begin{equation}
\label{equ9}
S\cap W=\S(g')\cap W.
\end{equation}
The function $g'$ satisfies in addition the following property 
\begin{equation}
\label{equ10}
\S(g')\subset S.
\end{equation}
Assume $g'(x)>0$ then $g(x)>0$ and moreover if $x\not\in S$ then
$x\in F$ and we get a contradiction since then the sign of $g'(x)$ would be
the sign of $-q^2(x)$. We have proved (\ref{equ10}).

Set $C=\overline{S}^{eucl}\setminus \S(g')$. Let $t\in\SR^0(X)$ be such
that $\Z(t)=\overline{\Z(p)\cap C}^{\Co}$. We clearly have $\Z(p)\cap
C\subset \Z(t^2 g').$ By Lemma \ref{hl3}, there exist $p'\in\Pol(X)$
positive definite on $X$ and a positive odd integer $l'$ such that
$f=p'p+t^{2l'}g'^{l'}$ is regulous on $X$ and satisfies $\Lambda(<f>)=\Lambda(<p>)$ on $C$.

Assume $x\not\in S$. We have $p(x)\leq 0$ since $W\subset
\Z(p)$. We have $g'(x)\leq 0$ by (\ref{equ10}). Hence $f(x)\leq 0$ and we have proved that 
\begin{equation}
\label{equ11} S(f)\subset S.
\end{equation}

Assume $x\in S\setminus W$. We have $p(x)> 0$. If $g'(x)\geq 0$ then clearly
$f(x)>0$. If $g'(x)< 0$
then $x\in C$ and $f(x)>0$ since $\Lambda(<f>)=\Lambda(<p>)$ on $C$.
We have proved that 
\begin{equation}
\label{equ12}
S\setminus W\subset \S(f)\setminus W.
\end{equation}

Since $W\subset \Z(p)$ and using (\ref{equ9}) it
follows that 
\begin{equation}
\label{equ13}
(S\cap W)\setminus (\S(f)\cap W)\subset
\Z(t)=\overline{\Z(p)\cap C}^{\Co}.
\end{equation}

We prove now that 
\begin{equation}
\label{equ14}
\Z(p)\cap C\subset \Bd(S).
\end{equation} 
Assume $y\in
\Z(p)\cap C\cap W=W\cap C$. We have $p(y)=0$, $y\in
\overline{S}^{eucl}\cap W$ and $g'(y)\leq 0$. We have $y\not\in S\cap
W$ by (\ref{equ9}). Hence $y\in
\Bd(S)\cap W$.\\
Assume $y\in \Z(p)\cap C$ and $y\not\in W$. Since
$p(y)=0$ and $y\not\in W$ then $y\not\in S$. We get $y\in
\overline{S}^{eucl}$ since $y\in C$ and it proves (\ref{equ14}).

From (\ref{equ11}), (\ref{equ12}), (\ref{equ13}) and (\ref{equ14}) it follows that $$S\setminus \S(f)\subset
\overline{\Z(p)\cap C}^{\Co}\cap W\subset \overline{\Bd(S)}^{\Co}\cap
W.$$
Since $S\cap\overline{\Bd(S)}^{\Co}=\emptyset$ we finally get
$$S=\S(f).$$
\end{proof}

\begin{thm}
\label{principal4}
Let $S$ be a
semi-algebraic subset of $X$. The set $S$ is regulous principal open if and only if we
have:\\
1) for any real algebraic subset $V$ of $X$ then $S\cap V$ is generically
principal,\\
and\\
2) $S\cap\overline{\Bd(S)}^{\Co}=\emptyset$.
\end{thm}

\begin{proof}
If $S=\S(f)$ with $f\in\SR^0(X)$ then we have already seen that
$S\cap\overline{\Bd(S)}^{\Co}=\emptyset$. Moreover, $S\cap V$ is generically
principal for any real algebraic subset $V$ of $X$ since
$f|_V\in\SR^0(V)$ and thus $\S(f|_V)$ coincides with a principal open
semi-algebraic subset of $V$ on $V\setminus \pol(f|_V)$.

Assume now the set $S$ satisfies the conditions 1) and 2) of the
theorem. We denote the set $X$ by $X_0$ and $S$ by $S_0$. Since $S_0$
is generically principal there exist $p_0\in\Pol(X_0)$ and an
algebraic subset $X_1$ of $X_0$ of positive codimension such that
$S_0\setminus X_1=\S(p_0)\setminus X_1$. If $X_1=\emptyset$ then we are
done since $S$ is even principal. If $X_1\not=\emptyset$
then we denote by $S_1$ the set $S_0\cap X_1$. Remark that $S_1$
satisfies the conditions 1) and 2) of the theorem as an open
semi-algebraic subset of $X_1$ and we can repeat the process used for
$S_0$ but here for the set $S_1$. So we build a finite sequence 
$$((X_0,S_0,p_0),\ldots,(X_k,S_k,p_k))$$
such that for $i=0,\ldots,k-1$, $X_{i+1}$ is an algebraic subset of $X_{i}$ of positive
codimension, $S_i=S\cap X_i$ satisfies the conditions 1) and 2), $p_i\in\Pol(X)$, $S_{i}\setminus
X_{i+1}=(\S(p_{i})\cap X_{i})\setminus X_{i+1}$ and $S_k=S\cap X_k=\S(p_k)\cap
X_k$ with $p_k\in\Pol(X)$.
By Proposition \ref{recolle}, there exists $g_{k-1}\in\SR^0(X)$ such that
$S_{k-1}=\S(g_{k-1})\cap X_{k-1}$. By successive application of Proposition
\ref{recolle}, there exists $g_{i}\in\SR^0(X)$ such that
$S_{i}=\S(g_{i})\cap X_{i}$ for $i=k-2,\ldots,0$, which establishes in
particular that $S$ is regulous principal open.
\end{proof}

\section{Lengths of signs of regulous functions (part 2)}

\subsection{Upper bounds for the lengths of signs of regulous functions}

We can use Corollary \ref{principal3} to improve some of the results
of Section $4$ concerning the lengths of signs of regulous
functions.

We extend the result of Proposition \ref{signsadim1}  which concerns
curves, to any real algebraic set of any dimension. It corresponds to
Theorem B of the introduction.
\begin{thm}
\label{Z-closed2}
Let $0\not=f\in \SR^0(X)$. Then $\Z(f)$ is Zariski closed if and only if
$\length(f)=1$.
\end{thm}

\begin{proof}
The proof of the ``if'' is trivial.

Assume $\Z(f)$ is Zariski closed. By Corollary \ref{principal3}, there
exist $p_1,p_2$ in $\Pol(X)$ such that $\S(f)=\S(p_1)$ and
$\S(-f)=\S(p_2)$. We conclude using Theorem
\ref{signsa}.
\end{proof}

\begin{cor}
\label{nbsignpol2}
Let $f\in
\SR^0(X)$, $k=\poldepth(f)$ and $((f_0,X_0,p_0,q_0),\ldots,(f_k,X_k,p_k,q_k))$
a ``polar sequence'' associated to $f$. Let 
$$t=\min\{i\in\{0,\ldots,k\}|\,\Z(f)\cap X_i\,{\rm is\,Zariski\,closed}\}.$$
Therefore, $$\length(f)\leq 1+2t.$$
\end{cor}

\begin{proof}
The proof is straightforward using Proposition \ref{nbsignpol} and
Theorem \ref{Z-closed2}.
\end{proof}

By the following proposition, we will improve the results of Theorems
\ref{nbsigngen} and
\ref{nbsignnormal}.
\begin{prop}
\label{nbsign2-3}
Let $f\in\SR^0(X)$ such that $\dim(\pol(f))\leq 1$.
Then $$\length(f)\leq 3.$$
More precisely, if $f=\dfrac{p}{q}$ on
$\dom(f)$, $p,q\in\Pol(X)$, $\Z(q)=\pol(f)$, then there exist
$h,r\in\Pol(X)$ such that $\Lambda(<f>)=\Lambda(<pq>\perp
<1,-r^2>\otimes <h>)$
on $X$.
\end{prop}

\begin{proof}
Let $f\in
\SR^0(X)$ such that $\dim(\pol(f))\leq 1$. We
get the proof, using Corollary \ref{nbsignpol2}, if $\Z(f)\cap
\pol(f)$ is Zariski closed (it is automatically the case when $\dim(\pol(f))<1$). So we assume $\dim
(\pol(f))=1$ and $\Z(f)\cap \pol(f)$ is not Zariski closed. 
We write $f=\dfrac{p}{q}$ on
$\dom(f)$ with $p,q\in\Pol(X)$ and $\Z(q)=\pol(f)$. We
decompose $\Z(q)=\pol(f)$ as a union $C_1\cup\cdots\cup C_t\cup W$ where
the $C_i$ are irreducible real algebraic curves and $\dim(W)=0$.
For each curve $C_i$, we denote by $F_i$ the regulous
closed set $\overline{(C_i)_{reg}}^{\Co}=\overline{(C_i)_{reg}}^{eucl}$. By \cite[Thm. 6.7]{FHMM}, the sets
$F_i$ are $\Co$-irreducible and $C_i\setminus F_i$ is empty or a
finite set of points. Since $\Z(f)\cap \pol(f)$ is not Zariski
closed, we have $\dim(\Z(f)\cap \pol(f))=1$. Since the $F_i$ are
$\Co$-irreducible, we get that $F_i\subset \Z(f)$ if and only if
$\dim(\Z(f)\cap C_i)=1$. Remark that there exists at least one $F_i$
such that $F_i\subset \Z(f)$ but $C_i\not\subset \Z(f)$ since
$\Z(f)\cap \pol(f)$ is not Zariski closed.
If $F_i\subset \Z(f)$ then $\Lambda(<f>)=\Lambda(<pq>)$ on $C_i$
outside a finite number of points. If $F_i\not\subset \Z(f)$ then
$\Z(f)\cap C_i$ is Zariski closed. It follows that there exists a real
algebraic subset $Y$ of $\pol(f)$ such that $\Z(f)\cap Y$ is Zariski
closed and such that $\Lambda(<f>)=\Lambda(<pq>)$ on $X\setminus Y$.
By Theorem \ref{Z-closed2}, there exists $h\in\Pol(X)$ such that
$\Lambda(<f>)=\Lambda(<h>)$ on $Y$. Let $r\in\Pol(X)$ be such that
$\Z(r)=Y$. The proof is done since 
$$\Lambda(<f>)=\Lambda(<pq>\perp <1,-r^2>\otimes <h>)\,{\rm on}\,X.$$
\end{proof}

\begin{rem}
Using Proposition \ref{nbsign2-3}, we recover the result of
Proposition \ref{signdim=1}: Let $X$ be an irreducible algebraic curve
and let $f\in\SR^0(X)$ such that $\Z(f)$ is not Zariski closed. By
Proposition \ref{nbsign2-3}, if $f=\dfrac{p}{q}$ on
$\dom(f)$, $p,q\in\Pol(X)$, $\Z(q)=\pol(f)$, then there exist
$h,r\in\Pol(X)$ such that $\Lambda(<f>)=\Lambda(<pq>\perp
<1,-r^2>\otimes <h>)$
on $X$. Since $\dim \Z(f)=1$ then $pq=0$ on $X$ (i.e $p=0$ and
$f$ is a continuous extension to $X$ of the null rational function) and
thus $\length(f)\leq 2$.
\end{rem}

As announced, we improve the upper bounds on
$\length$ given in Theorems \ref{nbsigngen} and \ref{nbsignnormal}.

\begin{thm}
\label{nbsigngen2}
Let $X\subset\R^n$ be a real algebraic set of dimension $d$. Then
$$\length(\SR^0(X))=1\,\,{\rm if}\,d=0,$$
$$\length(\SR^0(X))\leq 3\,\,{\rm if}\,d=1,$$
$$\length(\SR^0(X))\leq 2d-1\,\,{\rm else}.$$
\end{thm}

\begin{proof}
By Proposition \ref{nbsign2-3}, we are
left to prove the theorem for $d>2$. Let $f\in\SR^0(X)$. By
Proposition \ref{nbsignpol}, we can assume that $1+2\poldepth(f)>2d-1$
i.e $\poldepth(f)=d$.\\ 
Let $((f_0,X_0,p_0,q_0),\ldots,(f_d,X_d,p_d,q_d))$ be 
a polar sequence associated to $f$. For $i=0,\ldots,d$, we have $\dim
X_i=d-i$. In particular $\dim X_{d-2}=2$ and by Proposition
\ref{nbsign2-3}
there exist $h,r\in\Pol(X)$ such that $\Lambda(<f_{d-2}>)=\Lambda(<p_{d-2}q_{d-2}>\perp
<1,-r^2>\otimes <h>)$
on $X_{d-2}$. Then
$$\Lambda(<f>)=\Lambda(<p_0q_0>\perp_{i=1}^{d-2}  (<1,-\prod_{j=0}^{i-1}q_{j}^2>\otimes
<p_iq_i>) \perp <1,-\prod_{j=0}^{d-3}q_{j}^2r^2>\otimes
<h> )$$ on $X$ and the proof is done.
\end{proof}

\begin{thm}
\label{nbsignnormal2}
Let $X\subset\R^n$ be a real algebraic set of dimension $d$ such that
$\codim(\Sing(Y))>1$ for any irreducible component $Y$ of $X$. Then
$$\length(\SR^0(X))=1\,\,{\rm if}\,d=0\,{\rm or}\,1,$$
$$\length(\SR^0(X))\leq 3\,\,{\rm if}\,d=2$$
$$\length(\SR^0(X))\leq 2d-3\,\,{\rm else}.$$
\end{thm}

\begin{proof}
For $d\leq 2$ the theorem follows from Theorem
\ref{nbsignnormal}. For $d=3$ the theorem follows from Proposition
\ref{nbsign2-3} and Corollary \ref{cornormal}.

Assume $d\geq 4$. Let $f\in\SR^0(X)$. By Corollary \ref{cornormal}, we
have $\poldepth(f)\leq d-1$. By Theorem \ref{nbsignnormal} and
Proposition \ref{nbsignpol}, we get $1+2\poldepth(f)>2d-3$ i.e
$\poldepth(f)=d-1$. Let
$((f_0,X_0,p_0,q_0),\ldots,(f_{d-1},X_{d-1},p_{d-1},q_{d-1}))$ be 
a polar sequence associated to $f$.
By Corollary \ref{cornormal}, we
have $\dim(\pol(f))\leq d-2$ and thus it follows that for
$i=1,\ldots,d-1$ we have $\dim
X_i=d-i-1$. In particular $\dim X_{d-3}=2$ and by Proposition
\ref{nbsign2-3}
there exist $h,r\in\Pol(X)$ such that $\Lambda(<f_{d-3}>)=\Lambda(<p_{d-3}q_{d-3}>\perp
<1,-r^2>\otimes <h>)$
on $X_{d-3}$. Then
$$\Lambda(<f>)=\Lambda(<p_0q_0>\perp_{i=1}^{d-3}  (<1,-\prod_{j=0}^{i-1}q_{j}^2>\otimes
<p_iq_i>) \perp <1,-\prod_{j=0}^{d-4}q_{j}^2r^2>\otimes
<h> )$$ on $X$ and the proof is done.
\end{proof}

\begin{ex}
\label{cartan}
{\rm 
Consider $f=z-\dfrac{x^3}{x^2+y^2} \in\SR^0(\R^3)$. So
$\Z(z-\dfrac{x^3}{x^2+y^2} )\subset \R^3$ is the ``canopy'' of the
  Cartan umbrella $V=\Z(z(x^2+y^2)-x^3)\subset\R^3$. 
\begin{figure}[ht]
\centering
\includegraphics[height =6cm]{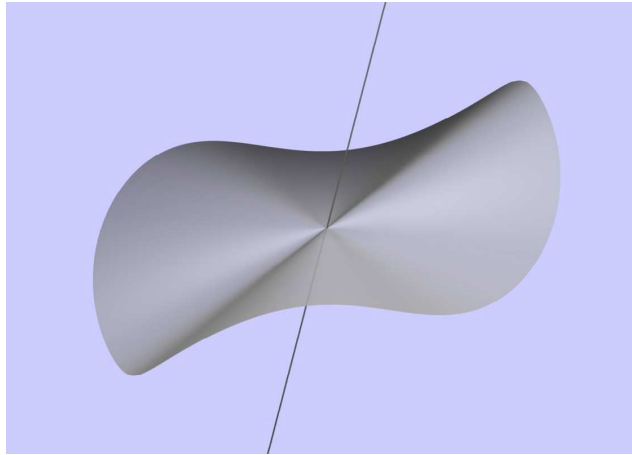}
\caption{Cartan umbrella.}
        \label{fig.cartan}
\end{figure}
Moreover, $\pol(f)$ is the stick of the umbrella and $\Z(f)\cap
\pol(f)=\{(0,0,0)\}$. According to Corollary \ref{nbsignpol2} we
get: $$\Lambda(<f>)=\Lambda(<(x^2+y^2)f>\perp <1,-x^2-y^2>\otimes
<z>)$$ on $\R^3$. Remark that since $\Z(f)$ is not Zariski closed then
$\length (f)>1$ (Theorem \ref{Z-closed2}). If
$\Lambda(<f>)=\Lambda(<p_1,p_2>)$ on $\R^3$ with
$p_1,p_2\in\Pol(\R^3)$ then it is easy to see that the product
$p_1p_2$ vanishes identically on $\R^3$. It follows that the form
$<p_1,p_2>$ is isotropic, a contradiction because $\length (f)>1$.
Hence $\length (f)=3$ and $\rho(f)=<(x^2+y^2)f>\perp <1,-x^2-y^2>\otimes
<z>$. From Theorem \ref{nbsignnormal2},
it follows that $\length(\SR^0(\R^3))=3$ and the bound given in
Theorem \ref{nbsignnormal2} is optimal for $d=3$.}
\end{ex}

\subsection{Characterization of regulous functions with length of sign
  equal to one}

Let $X\subset\R^n$ be a real algebraic set. By Theorem \ref{Z-closed2},
we know that a non-zero regulous function on $X$ has a length of sign equal to
one if and only if its zero set is Zariski closed.

We give some several equivalent
characterizations of regulous functions with length of sign equal to
one for central and irreducible real algebraic sets.
\begin{prop}
\label{complementary}
Let $X\subset\R^n$ be a central and irreducible real algebraic set. Let $0\not=f\in\SR^0(X)$.
The following properties are equivalent:\\
a) $\length(f)=1$.\\
b) $\Z(f)$ is Zariski closed.\\
c) $\S(f^2)=\S(f)\cup\S(-f)=X\setminus \Z(f)$ is principal.\\
d) $\S(f^2)\cap\overline{\Bd(\S(f^2))}^{Zar}=\emptyset$.
\end{prop}

\begin{proof}
Equivalence between a) and b) (resp. c) and d)) is Theorem
\ref{Z-closed2} (resp. Theorem \ref{principal2}) and the assumptions
that $X$ is central and irreducible are not required. It is clear
that b) implies c). We are reduced to proving c) implies b). Assume $\S(f^2)=X\setminus \Z(f)$ is principal,
namely $\S(f^2)=\S(p)$ for $p\in\Pol(X)$. Since $f\not=0$ then
$\Z(f)$ is a proper subset of $X$. Since $\overline{X_{reg}}^{eucl} =X$ ($X$ is central) and
$X$ is irreducible then it follows from \cite[Prop. 6.6]{FHMM} that
$\dim \Z(f)<\dim X$.
Notice that $\S(-p)\subset \Z(f)$. If $\S(-p)\not=\emptyset$ then we claim
that $\dim \S(-p)=\dim X$: Let $\tilde{\S}(-p)$ be the constructible
subset of $\Sp_r \Pol(X)$ associated to $\S(-p)$ (see
\cite[Ch. 7]{BCR}). We have $\dim \S(-p)=\dim \tilde{\S}(-p)$
\cite[Prop. 7.5.6]{BCR}. Since $X$ is central and $ \S (-p)$ is non-empty and open then $\S(-p)\cap X_{reg}\not=\emptyset$. By \cite[Prop. 7.6.2]{BCR}, 
$\tilde{\S}(-p)\cap\Sp_r \K(X)\not=\emptyset$ 
and we get $\dim
\tilde{\S}(-p)=\dim X$ \cite[Prop. 7.5.8]{BCR} which gives the claim.
It follows from the claim and above remarks that $\S(-p)=\emptyset$ and thus 
$\Z(f)=\Z(p)$ is
Zariski closed.
\end{proof}

\begin{cor}
\label{corcomplementary}
Let $X\subset\R^n$ be a central and irreducible real algebraic set. Let $0\not=f\in\SR^0(X)$ such
that $\S(f)$ is principal and $f$ is nonnegative on $X$. Then
$\length(f)=1$.
\end{cor}

\begin{ex}
{\rm The assumption that $X$ is central in 
  Proposition \ref{complementary} and Corollary \ref{corcomplementary}
  is a necessary assumption. Consider the regulous function
  $f=1-\dfrac{x^3}{x^2+y^2}$ restricted to $X=\Z(y^2-x^3+x^2)$ of Example \ref{exemple1}, $f$ is
  non-negative on $X$, $\S(f)\cap X$ is principal ($\S(f)\cap
  X=\S(1-x)\cap X$) but $\Z(f)\cap X$ is
  not Zariski closed. The assumption that $X$ is irreducible is also a
necessary assumption, see Example \ref{exdim1red}.}
\end{ex}

\begin{ex}
{\rm 
We have already seen that if $f$ is a regulous function on a real
algebraic set $X$ then the property that $\Z(f)$ is Zariski closed
(condition 1) of Theorem \ref{signsa}) implies that $\S(f)$ and
$\S(-f)$ are both principal (conditions 2) and 3) of Theorem
\ref{signsa}). We prove now that the converse is not always true even
if $X$ is central and irreducible.
Consider the following regulous functions on the plane:
$h=(1-\dfrac{x^3}{x^2+y^2})^2$,
$g=-(y^2+(x+\dfrac{1}{2})(x-\dfrac{1}{2})(x-4)(x-5))$, $f=hg$.
We have $\Bd(\S(f))=\Z(g)=\overline{\Bd(\S(f))}^{Zar}$, hence $\S(f)$
is principal (Theorem \ref{principal2}) and more precisely
$\S(f)=\S(g)$. We have $\Bd(\S(-f))=\Z(g)\cup \Z(h)$, hence
$\overline{\Bd(\S(-f))}^{Zar}=\Z(g)\cup\Z((x^2+y^2)^2h)=\Z(g)\cup\Z(h)\cup\{(0,0)\}$.
Since $g$ and $f$ are both positive at the origin then
$\overline{\Bd(\S(-f))}^{Zar}\cap \S(-f)=\emptyset$ and thus $\S(-f)$
is principal; more precisely $\S(-f)=\S(-g(x^2+y^2)^2h)$. We also have 
$\S(f^2)\cap\overline{\Bd(\S(f^2))}^{Zar}=\{(0,0)\}$ and thus $\Z(f)$ is
not Zariski closed (Proposition \ref{complementary}).}
\end{ex}

In the previous example, the problems arise in part because of the
$\Co$-reducibility of the
zero set of the regulous function $f$.
\begin{prop}
\label{cond2et3}
Let $X\subset\R^n$ be a central and irreducible real algebraic set of dimension $d$. Let $f\in\SR^0(X)$ be such that $\S(f)$ is principal, $\S(-f)$ is
principal, $\Z(f)$ is $\Co$-irreducible and $\Bd (\S(f))\cap \Bd (\S(-f))\not=\emptyset$. Then $\Z(f)$ is Zariski
closed.
\end{prop}

\begin{proof}
The sets $\S(-f)$ and $\S(f)$ are both non-empty since $\Bd
(\S(f))\cap \Bd (\S(-f))\not=\emptyset$.
As we have already explained in the proof of Proposition
\ref{complementary} and since $X$ is central and irreducible, we have $\dim \S(f)=\dim \S(-f)=d$. 
We claim that $\dim \Bd(\S(f))=d-1$. There exist $x\in X_{reg}$ and a semi-algebraic neighbourhood $U$ of $x$ in $X$
satisfying the following three properties:\\
$\bullet$ There exists a semi-algebraic homeomorphism from $U$ onto a
semi-algebraic $U'$ of the origin in $\R^d$ (mapping $x$ to the
origin).\\
$\bullet$ $\S(f)\cap U\not=\emptyset$.\\
$\bullet$ $(X\setminus\overline{\S(f)}^{eucl})\cap U\not=\emptyset$.\\
The first property follows from \cite[Prop. 3.3.11]{BCR}. The second
and the third properties are consequences of the assumption $\Bd (\S(f))\cap \Bd (\S(-f))\not=\emptyset$ and also because $X$ is central and
irreducible. Since $\Bd(\S(f))\cap
U=U\setminus( (\S(f)\cap
U)\cup((X\setminus\overline{\S(f)}^{eucl})\cap U))$, we get $\dim
\Bd(\S(f))\geq d-1$ applying \cite[lem. 4.5.2]{BCR}. Since $X$ is
irreducible and central then $\dim \Z(f)\leq d-1$
(\cite[Prop. 6.6]{FHMM}). Since $\Bd(\S(f))\subset \Z(f)$, we get the
claim and moreover we see that $\dim\Z(f)=d-1$.

By the same arguments we get $\dim \Bd(\S(-f))=d-1$. 
Since $X$ is
irreducible and central and since $\dim \Z(f)\leq d-1$ then
$X=\overline{X\setminus \Z(f)}^{eucl}$ and thus $\Z(f)=\Bd(\S(f))\cup\Bd(\S(-f))$.
Since $\Z(f)=\Bd(\S(f))\cup\Bd(\S(-f))$, $\dim \Z(f)=\dim \Bd(\S(f))=\dim \Bd(\S(-f))=d-1$ and
since by assumption $\Z(f)$ is $\Co$-irreducible then we get
$$\Z(f)=\overline{\Bd(\S(f))}^{\Co}=\overline{\Bd(\S(-f))}^{\Co}.$$
Hence
$\overline{\Z(f)}^{Zar}=\overline{\Bd(\S(f))}^{Zar}=\overline{\Bd(\S(-f))}^{Zar}$
and thus
$\overline{\Bd(\S(f^2))}^{Zar}=\overline{\Z(f)}^{Zar}=\overline{\Bd(\S(f))}^{Zar}=\overline{\Bd(\S(-f))}^{Zar}$.
Since $\S(f)$ is principal then $\overline{\Bd(\S(f^2))}^{Zar}\cap
\S(f)=\emptyset$. Since $\S(-f)$ is principal then $\overline{\Bd(\S(f^2))}^{Zar}\cap
\S(-f)=\emptyset$. Hence $\overline{\Bd(\S(f^2))}^{Zar}\cap
\S(f^2)=\emptyset$ and the proof is done (Proposition
\ref{complementary}).
\end{proof}

\subsection{Complexity of regulous principal semi-algebraic sets}

\begin{thm} \cite[Prop. and Def. 3.7 Ch. 1]{ABR}, \cite[Thm. 2.8]{MP1}\\
Let $X\subset\R^n$ be a real algebraic set of dimension $d$.  The
cokernel of the inclusion map $\A(X)\subset \F(X)$ has two primary
torsion and moreover $$2^d\F(X)\subset\A(X).$$
\end{thm}

From the previous theorem, we can introduce some invariants of
semi-algebraic sets (see \cite[Prop. and Def. 3.7 Ch. 1]{ABR} for the
original definitions).
\begin{defn}
Let $X\subset\R^n$ be a real algebraic set. Let $C$ be a non-empty
semi-algebraic subset of $X$.\\
$\bullet$ The minimal number $k>0$ such that $k\un_C\in\A(X)$ is a power of two,
say $k=2^{\w(C)}$.\\
$\bullet$ There exists a form $\rho$ over $\Pol(X)$ such
that $\Lambda(\rho)=2^{\w(C)}\un_C$. The form $\rho$ can always be chosen anisotropic
and then it is unique. We denote by $\rho(C)$ the corresponding
anisotropic form and by $\ell(C)$ the dimension of $\rho(C)$.\\
$\bullet$ The
number $\w(C)$ is called the width of $C$, the number $\ell(C)$ is
called the length of $C$ and the form $\rho(C)$ is called the defining
form of $C$.
\end{defn}

\begin{cor} \cite[Thm. 2.8]{MP1}\\
Let $X\subset\R^n$ be a real algebraic set of dimension $d$. Let $C$ be a non-empty
semi-algebraic subset of $X$. Then $$\w(C)\leq d.$$
\end{cor}

The following proposition characterizes the widths of regulous closed
sets and regulous principal sets.
\begin{prop}
\label{width}
Let $X\subset\R^n$ be a real algebraic set. Let $0\not=f\in\SR^0(X)$. In
case the considered set is non-empty, we get:\\ 
$\bullet$ $\w(\Z(f))=0$,
$\w(X\setminus\Z(f))=0$, $\w(\S(f))\leq 1$ and $\w(\bar{\S}(f))\leq 1$.\\
$\bullet$ If $f$ is non-negative on $X$ then
$\w(\S(f))=\w(\bar{\S}(f))=0$.\\
$\bullet$ We have $\w(\S(f))=\w(\S(-f))$ in case
$\S(f)$ and $\S(-f)$ are both non-empty.\\
$\bullet$ We have $\w(\bar{\S}(f))=\w(\S(f))$ in case
$\bar{\S}(f)$ and $\S(f)$ are both non-empty.
\end{prop}

\begin{proof}
We have $\Lambda(<1>\perp \rho(-f^2))=\un_{\Z(f)}$,
$\Lambda(\rho(f^2))=\un_{X\setminus\Z(f)}$,
$\Lambda(\rho(f)\perp\rho(f^2))=2\un_{\S(f)}$ and $\Lambda(<1>\perp
\rho(f)\perp <1>\perp \rho(-f^2))=2\un_{\bar{\S}(f)}$.

If $f$ is non-negative on $X$ then $\Lambda(\rho(f))=\un_{\S(f)}$ and
$\Lambda(<1>)=\un_{\bar{\S}(f)}$.

Assume $\S(f)$ and $\S(-f)$ are both non-empty.  If $\w(\S(-f))=0$ then 
$\Lambda(<-1>\otimes \rho(\S(-f))\perp <1>\perp <-1>\otimes
\rho(\Z(f)))=\un_{\S(f)}$ if $\Z(f)\not=\emptyset$ and 
$\Lambda(<-1>\otimes \rho(\S(-f))\perp <1>)=\un_{\S(f)}$
if $\Z(f)=\emptyset$. It follows that $\w(\S(f))=0$.

Assume $\S(f)$ and $\bar{\S}(f)$ are both non-empty.  If $\w(\S(f))=0$ then 
$\Lambda(\rho(\S(f))\perp
\rho(\Z(f)))=\un_{\bar{\S}(f)}$ if $\Z(f)\not=\emptyset$ and 
$\Lambda(\rho(\S(f)))=\un_{\bar{\S}(f)}$
if $\Z(f)=\emptyset$. It follows that $\w(\bar{\S}(f))=0$.
If $\w(\bar{\S}(f))=0$ then 
$\Lambda(\rho(\bar{\S}(f))\perp <-1>\otimes
\rho(\Z(f)))=\un_{\S(f)}$ if $\Z(f)\not=\emptyset$ and 
$\Lambda(\rho(\bar{\S}(f)))=\un_{\S(f)}$
if $\Z(f)=\emptyset$. It follows that $\w(\S(f))=0$ and the proof is done.
\end{proof}

We compare the lengths of regulous closed
sets and regulous principal sets and the lengths of the signs of
regulous functions.
\begin{prop}
\label{lengthsa}
Let $X\subset\R^n$ be a real algebraic set. Let $0\not=f\in\SR^0(X)$. In
case the considered set is non-empty, we get:\\
$\bullet$ $\ell(\Z(f))\leq 1+\ell(f^2)\leq 1+\ell(f)^2$ and
$\rho(\Z(f))$ is the anisotropic form similar to $<1>\perp
\rho(-f^2)$.\\
$\bullet$ $\ell(X\setminus\Z(f))=\ell(f^2)$ and
$\rho(X\setminus\Z(f))=\rho(f^2)$.\\
$\bullet$ If $f$ is non-negative on $X$ then $\ell(\S(f))=\ell(f)$ and 
$\rho(\S(f))=\rho(f)$.\\
$\bullet$ If $\w(\S(f))=1$ then $\ell(\S(f))\leq \ell(f)+\ell(f^2)\leq
\ell(f)(1+\ell(f))$ and $\rho(\S(f))$ is the anisotropic form similar
to $\rho(f)\perp\rho(f^2)$.\\
$\bullet$ If $f$ is non-negative on $X$ then $\ell(\bar{\S}(f))=1$ and 
$\rho(\S(f))=<1>$.\\
$\bullet$ If $\w(\bar{\S}(f))=1$ then $\ell(\bar{\S}(f))\leq 2+\ell(f)+\ell(f^2)\leq
2+\ell(f)(1+\ell(f))$ and $\rho(\bar{\S}(f))$ is the anisotropic form similar
to $<1,1>\perp \rho(f)\perp\rho(-f^2)$.\\
$\bullet$ If $\S(f)$ and $\S(-f)$ are both non-empty and if
$\w(\S(f))=0$ then $\ell(f)\leq \ell(\S(f))+\ell(\S(-f))$ and $\rho(f)$ is the
anisotropic form similar to $\rho(\S(f))\perp <-1>\otimes
\rho(\S(-f))$.\\
$\bullet$ If $\S(f)$ and $\S(-f)$ are both non-empty and if
$\w(\S(f))=1$ and $\Z(f)\not=\emptyset$ then\\ $\ell(f)\leq \inf\{\ell(\S(f)),\ell(\S(-f))\}+\ell(\Z(f))+1$ and $\rho(f)$ is the
anisotropic form similar to $\rho(\S(f))\perp <-1>\perp
\rho(\Z(f))$ and $<-1>\otimes \rho(\S(-f))\perp <1>\perp
<-1>\otimes\rho(\Z(f))$.\\
$\bullet$ If $\S(f)$ and $\S(-f)$ are both non-empty and if
$\w(\S(f))=1$ and $\Z(f)=\emptyset$ then\\ $\ell(f)\leq \inf\{\ell(\S(f)),\ell(\S(-f))\}+1$ and $\rho(f)$ is the
anisotropic form similar to $\rho(\S(f))\perp <-1>$ and $<-1>\otimes \rho(\S(-f))\perp <1>$.
\end{prop}

\begin{proof}
Note that trivially $\ell(f)=\ell(-f)$ and $\ell(f^2)\leq \ell(f)^2$
since $\Lambda(\rho(f)\otimes\rho(f))=\Lambda(\rho(f^2))$ on $X$. 
We give the proof of the last three statements. Assume $\S(f)$ and
$\S(-f)$ are both non-empty. By Proposition \ref{width} we know that $\w(\S(f))=\w(\S(-f))$.
If $\w(\S(f))=0$ then verify that
$\Lambda(\rho(\S(f))\perp <-1>\otimes
\rho(\S(-f)))=\Lambda (<f>)$ on $X$.
If $\w(\S(f))=1$ and $\Z(f)\not=\emptyset$ then verify that $\Lambda(\rho(\S(f))\perp <-1>\perp
\rho(\Z(f)))=\Lambda(<-1>\otimes \rho(\S(-f))\perp <1>\perp
<-1>\otimes\rho(\Z(f))))=\Lambda (<f>)$ on $X$. If $\w(\S(f))=1$ and
$\Z(f)=\emptyset$ then we can remove the form $\rho(\Z(f))$ from
the above formulas.
The rest of the proof follows essentially from the
arguments given in the proof of Proposition \ref{width}.
\end{proof}

\begin{rem} 
The reader may find more general upper bounds for the length of
semi-algebraic sets in \cite[Thm. 2.5, Ch. 5]{ABR}. Note that the result
given in \cite[Rem. 2.6, Ch. 5]{ABR} seems to be incorrect. Consider
the set $X=\{(0,0)\}\sqcup F$ of Example \ref{exemple1} and let
$f=x$ restricted to $X$. We have $\Z(f)=\{(0,0)\}$. We get
$\w(\Z(f))=0$ and $\ell(\Z(f))\leq 2$ since $\Lambda(<1,-x^2>)=\un_{\{0,0\}}$ (or use
Proposition \ref{width}). Since $\w(\Z(f))=0$, in \cite[Rem. 2.6,
Ch. 5]{ABR} they predict that $\ell(\Z(f))=1$ i.e there exists a
polynomial function that does not vanish at the origin and vanishing
identically on $F$, impossible. In this example,
$\ell(\Z(f))=2=1+\ell(f^2)$ (the bound given in the first statement of
Proposition \ref{lengthsa} is the best possible in this case).
\end{rem}

We may improve the result of Propositions \ref{width} and \ref{lengthsa} if we assume
that the regulous function changes of signs sufficiently.
\begin{prop}
\label{bonnedim}
Let $X\subset\R^n$ be an irreducible real algebraic set. Let
$f\in\SR^0(X)$ be such that $\dim\S(f)=\dim\S(-f)=\dim X$.
Then $\w(\S(f))=\w(\S(-f))=1$, $\ell(\Z(f))\geq 2$, $\ell(\S(f))\geq
2$ and $\ell(\S(-f))\geq 2$.
\end{prop}

\begin{proof}
Assume $\w (\S(f))=0$ and $\rho(\S(f))$ is the similarity class of the
anisotropic form $<p_1,\ldots,p_k>$, $p_1,\ldots,p_k\in\Pol(X)$.
We claim there exists $x\in\S(f)$ such that $p_i(x)\not= 0$ for
$i=1,\ldots, k$. Otherwise $\prod_{i=1}^k p_i$ vanishes identically on
$\S(f)$ and thus also on $X$ since by assumption
$\overline{\S(f)}^{Zar}=X$. Since $\Pol(X)$ is an integral domain then
it follows that $<p_1,\ldots,p_k>$ is isotropic, a contradiction. Since
$\sum_{i=1}^k\sign(p_i)(x)=1$, it follows that $k$ is odd. By the
above arguments, there exists $y\in\S(-f)$ such that $p_i(y)\not= 0$ for
$i=1,\ldots, k$ and it follows that $k$ is even. Using Proposition
\ref{width} we conclude that $\w(\S(f))=1$. Hence we get $\ell(\S(f))\geq
2$. Changing $f$ by $-f$ in the above arguments or using Proposition
\ref{width} we get $\w(\S(-f))=1$
and $\ell(\S(-f))\geq 2$.
Assume now that $\ell(\Z(f))=1$. There exists a non-zero $p\in\Pol(X)$
such that $\Lambda(<p>)=1$ on $\Z(f)$ and $\Lambda(<p>)=0$ on
$\S(f)\cup\S(-f)$. It impossible because $X$ is irreducible.
\end{proof}

\begin{prop}
Let $X\subset\R^n$ be a real algebraic set. Let $0\not=f\in\SR^0(X)$.
The following properties are equivalent:\\
a) $\length(f)=1$.\\
b) $\Z(f)$ is Zariski closed.\\
c) $\ell(X\setminus\Z(f))=1$.
\end{prop}

\begin{proof}
Equivalence between a) and b) is Theorem
\ref{Z-closed2}. 
Assume $\length(f)=1$. There exists $p\in\Pol(X)$ such that
$\Lambda(<p>)=\Lambda (<f>)$ on $X$. Thus
$\Lambda(<p^2>)=\un_{X\setminus\Z(f)}$ and so
$\ell(X\setminus\Z(f))=1$.
Assume $\ell(X\setminus\Z(f))=1$. Then clearly $\w(X\setminus\Z(f))=0$
and thus there exists $p\in\Pol(X)$ such that
$\Lambda(<p>)=\un_{X\setminus\Z(f)}$. Hence $\Z(f)=\Z(p)$ i.e $\Z(f)$
is Zariski closed.
\end{proof}

\begin{prop}
\label{longueurinfà2}
Let $X\subset\R^n$ be a real algebraic set. Let $f\in\SR^0(X)$.
Then $\S(f)$ is principal if $\ell(\S(f))\leq 2$.
\end{prop}

\begin{proof}
We assume $\S(f)$ is non-empty and $\ell(\S(f))\leq 2$. By Proposition
\ref{width} we have $\w(\S(f))\leq 1$. We study all the
possible couples $(\ell(\S(f)),\w(\S(f)))$.\\
$\bullet$ Assume $\ell(\S(f))=2$ and $\w(\S(f))=1$. There exist
$p,q\in\Pol(X)$ such that $\Lambda(<p,q>)=2\un_{\S(f)}$ and $<p,q>$ is
anisotropic. We clearly have $\S(f)\subset\S(p)$ and $\S(f)\subset\S(q)$.
We claim that $\Bd(\S(f))\subset \Z(pq)$: Otherwise we
may assume there exists $x\in\Bd(\S(f))$ such that $p(x)<0$ and
$q(x)>0$. Thus there exists $y\in\S(f)$ such that $p(y)<0$,
impossible.
Hence $\overline{\Bd(\S(f))}^{Zar}\subset \Z(pq)$. Since $\S(f)\subset
\S(p,q)$ then it follows that
$\S(f)\cap\overline{\Bd(\S(f))}^{Zar}=\emptyset$. By Theorem
\ref{principal2}, we conclude that $\S(f)$ is principal.\\
$\bullet$ The case $\ell(\S(f))=1$ and $\w(\S(f))=1$ is clearly
impossible.\\
$\bullet$ Assume $\ell(\S(f))=1$ and $\w(\S(f))=0$. There exists $p\in\Pol(X)$ such that
$\Lambda(<p>)=\un_{\S(f)}$ and thus $\S(f)=\S(p)$.\\
$\bullet$ Assume $\ell(\S(f))=2$ and $\w(\S(f))=0$. There exist
$p,q\in\Pol(X)$ such that $\Lambda(<p,q>)=\un_{\S(f)}$ and $<p,q>$ is
anisotropic. We clearly have $\S(f)\subset\bar{\S}(p)$ and
$\S(f)\subset\bar{\S}(q)$. Thus $\overline{\S(f)}^{eucl}\subset
\bar{\S}(p,q)$ and it follows that $\Bd(\S(f))\subset
\bar{\S}(p,q)$. Since $\Lambda (<p,q>)=0$ on $\Bd(\S(f))$ then 
we get $\Bd(\S(f))\subset \overline{\Bd(\S(f))}^{Zar}\subset
\Z(p)\cap\Z(q)$. Looking at the signature of the anisotropic form
$<p,q>$, it follows that
$\S(f)\cap\overline{\Bd(\S(f))}^{Zar}=\emptyset$. By Theorem
\ref{principal2}, the proof is done.
\end{proof}

\begin{thm}
Let $X\subset\R^n$ be a central and irreducible real algebraic set.
Let $f\in\SR^0(X)$.
Then $\S(f)$ is principal if and only if $\ell(\S(f))\leq 2$.
\end{thm}

\begin{proof}
Proposition \ref{longueurinfà2} gives one implication. 
One proves now the
converse implication. Assume $\S(f)\not=\emptyset$ and there exists $p\in\Pol(X)$ such that
$\S(f)=\S(p)$. If $f$ is non-negative on $X$ then
$\ell(f)=\ell(\S(f))=1$ by Corollary \ref{corcomplementary}. So we can
assume $\S(-f)\not=\emptyset$. Since $X$ is irreducible and central,
it follows that $\dim\S(f)=\dim\S(-f)=\dim X$. By Proposition
\ref{bonnedim}, we get $\w(\S(f))=1$. Since
$\Lambda(<p,p^2>)=2\un_{\S(f)}$ then the proof is done.
\end{proof}

\begin{rem}
The author cautions the reader that \cite[Cor. 2.2, Ch. 5]{ABR} seems
to be incorrect. Indeed, the width of a principal semi-algebraic set is not
always equal to one: $\w(\S(p))=0$ when $p$ is a non-negative
polynomial function on a real algebraic set.
\end{rem}

\section{Signs of regulous functions}

Throughout this section $X$ will denote a real algebraic subset of
dimension $d$ of $\R^n$. The goal of this
section is to characterize the signs of continuous semi-algebraic
functions that coincide with the signs of regulous functions. We deal
more particularly with the case where $X$ is central and 
the semi-algebraic functions are rational continuous.

The following statement is a regulous generalization of Lemma
\ref{hl1}.
\begin{lem}
\label{hl4}
Let $S$ be a closed semi-algebraic subset of $X$. Let $f,g\in\SR^0
(X)$. There exist $p\in\Pol (X)$ and $h\in\SR^0(X)$ such that $p>0$ on $X$, $h\geq 0$
on $X$, $\Lambda(<pf+hg>)=\Lambda(<f>)$ on $S$ and $\Z
(h)=\overline{\Z(f)\cap S}^{\Co}$.
\end{lem} 

\begin{proof}
As in the proof of Lemma \ref{hl3}, we may assume $S$ is a closed
semi-algebraic subset of $\R^n$ and $f,g\in\SR^0(\R^n)$.
Take $h\in\SR^0(\R^n)$ such that $\Z
(h)=\overline{\Z(f)\cap S}^{\Co}$. By
\cite[Thm. 2.6.6]{BCR}, for a sufficiently big positive even integer
$N$ the function $h^N\dfrac{g}{f}$ extended by $0$ on $\Z(f)$ is semi-algebraic and continuous on
$S$. The end of the proof is the same as that of Lemma \ref{hl3}.
\end{proof}

The following theorem is a regulous generalization of Theorem \ref{signsa}.
\begin{thm}
\label{signprincipaux}
Let $f$ be a continuous semi-algebraic function on $X$ satisfying the following
$3$ conditions:\\
$\bullet$ there exists $g_1\in\SR^0(X)$ such that $\S(f)=\S(g_1)$,\\
$\bullet$ there exists $g_2\in\SR^0(X)$ such that $\S(-f)=\S(-g_2)$,\\
$\bullet$ there exists $g_3\in\SR^0(X)$ such that $\Z(f)=\Z(g_3)$.\\
Then there exists $g\in\SR^0(X)$ such
that $\Lambda(<f>)=\Lambda(<g>)$ on $X$.
\end{thm}

\begin{proof}
Let
$S$ denote the set $\bar{\S}(f)$. The set $S$ is closed and
semi-algebraic since $f$ is respectively continuous and semi-algebraic.
Remark that $\S(f)=\S(g_1g_3^2)$ and $\Z(f)\subset \Z(g_1g_3^2)$ and
thus we get $\Lambda(<f>)=\Lambda(<g_1g_3^2>)$ on $S$. By Lemma \ref{hl4}, there exist $p\in\Pol (X)$ and $h\in\SR^0(X)$ such that $p>0$ on $X$, $h\geq 0$
on $X$, $\Lambda(<pg_1g_3^2+hg_2>)=\Lambda(<g_1g_3^2>)=\Lambda(<f>)$ on $S$ and $\Z
(h)=\overline{\Z(g_1g_3^2)\cap S}^{\Co}$. We denote by $g$ the
regulous function $pg_1g_3^2+hg_2$.
We are left to prove that $\Lambda(<g>)=\Lambda(<f>)$ on $\S(-f)$.
Let $x\not\in S$ i.e $f(x)<0$. Since $g_2(x)<0$ and $g_1(x)\leq
0$, it is sufficient to prove that $h(x)>0$. We have $S\cap \Z(g_1
g_3^2)\subset \Z(f)\cap  \Z(g_1 g_3^2)\subset \Z(f)=\Z(g_3)$ and thus 
$\Z(h)=\overline{\Z(g_1g_3^2)\cap S}^{\Co}\subset
\overline{\Z(g_3)}^{\Co}=\Z(g_3)=\Z(f)$. It follows that $h(x)>0$ and
the proof is done.
\end{proof}

\begin{prop}
\label{signratcont=signregu1}
Let $X\subset\R^n$ be a central real algebraic set and let $f\in\SR_0(X)$.
There exists $g\in\SR^0(X)$ such
that $\Lambda(<f>)=\Lambda(<g>)$ on $X$ if and only if $\Z(f)$ is
regulous closed and the
semi-algebraic sets $\S(f)\cap \pol(f)$ and $\S(-f)\cap \pol(f)$ are $\SR^0(\pol(f))$-principal.
\end{prop}

\begin{proof}
Let $0\not=f\in\SR_0(X)$, there exist $p,q\in\Pol(X)$ such that
$f=\dfrac{p}{q}$ on $X\setminus\pol(f)$ and $\Z(q)=\pol(f)$.

If there exists $g\in\SR^0(X)$ such
that $\Lambda(<f>)=\Lambda(<g>)$ on $X$ then clearly $\Z(f)$ is
regulous closed and the
semi-algebraic sets $\S(f)\cap \pol(f)$ and $\S(-f)\cap \pol(f)$ are $\SR^0(\pol(f))$-principal.

Assume for the rest of the proof that $\Z(f)$ is
regulous closed and the
semi-algebraic sets $\S(f)\cap \pol(f)$ and $\S(-f)\cap \pol(f)$ are
$\SR^0(\pol(f))$-principal. Since the restriction map
$\SR^0(X)\rightarrow\SR^0(\pol(f))$ is surjective there exist $g_1,g_2\in \SR^0(X)$
such that $\S(f)\cap\pol(f)=\S(g_1)\cap \pol(f)$ and
$\S(-f)\cap\pol(f)=\S(-g_2)\cap \pol(f)$.
By hypothesis, there exists $g_3\in \SR^0(X)$ such
that $\Z(g_3)=\Z(f)$. We have $\S(f)\cap
\overline{\Bd(\S(f))}^{\Co}\subset \S(f)\cap\overline{\Z(f)}^{\Co}=\S(f)\cap\overline{\Z(g_3)}^{\Co}
=\S(f)\cap\Z(g_3)=\S(f)\cap\Z(f)=\emptyset$. Since $\S(f)\setminus
\pol(f)=\S(pq)\setminus \pol(f)$, it follows from Proposition
\ref{recolle} that there exists $h_1\in \SR^0(X)$ such that
$\S(f)=\S(h_1)$. The same reasoning gives $h_2\in \SR^0(X)$ such that
$\S(-f)=\S(-h_2)$. Since $X$ is central then the function $f$ is
semi-algebraic. By Theorem \ref{signprincipaux} the proof is done.
\end{proof}

\begin{cor}
Let $X\subset\R^n$ be a central real algebraic set. Let $f\in\SR_0(X)$ such
that $\Z(f)$ is
regulous closed and $\dim (\pol(f))\leq 1$ (it is automatically the case if $\dim
X\leq 2$).
There exists $g\in\SR^0(X)$ such
that $\Lambda(<f>)=\Lambda(<g>)$ on $X$.
\end{cor}

\begin{proof}
The restriction of $f$ to $\pol(f)$ is a continuous semi-algebraic
function. So the sets $\S(f)\cap \pol(f)$ and $\S(-f)\cap \pol(f)$ are
open semi-algebraic subsets of $\pol(f)$. Now since
$\dim (\pol(f))\leq 1$ then the sets $\S(f)\cap \pol(f)$ and $\S(-f)\cap
\pol(f)$ are principal by \cite{Br}.  By Proposition
\ref{signratcont=signregu1} the proof is complete.
\end{proof}

\begin{prop}
\label{signratcont=signregu2}
Let $X\subset\R^n$ be a central real algebraic set and let $f\in\SR_0(X)$.
There exists $g\in\SR^0(X)$ such
that $\Lambda(<f>)=\Lambda(<g>)$ on $X$ if and only if $\Z(f)$ is
regulous closed and for any algebraic
subset $V$ of $X$ the
semi-algebraic sets $\S(f)\cap V$ and $\S(-f)\cap V$ are generically principal.
\end{prop}

\begin{proof}
By Theorem \ref{signprincipaux}, we only have to prove the ``if'' part.
Assume that $\Z(f)$ is
regulous closed and for any algebraic
subset $V$ of $X$ the
semi-algebraic sets $\S(f)\cap V$ and $\S(-f)\cap V$ are generically
principal. Since $\S(f)\cap
\overline{\Bd(\S(f))}^{\Co}=\emptyset$ and $\S(-f)\cap
\overline{\Bd(\S(-f))}^{\Co}=\emptyset$ (see the proof of Proposition
\ref{signratcont=signregu1}, it is a consequence of the hypothesis
that $\Z(f)$ is
regulous closed), it follows from Theorem
\ref{principal4} that there exist $g_1,g_2\in \SR^0(X)$
such that $\S(f)=\S(g_1)$ and
$\S(-f)=\S(-g_2)$. To end the proof use Theorem \ref{signprincipaux}.
\end{proof}

\end{document}